\documentclass{amsart}
\usepackage{amsmath,amssymb}
\usepackage{amsthm}
\usepackage[active]{srcltx}
\usepackage{color}
\usepackage{graphicx}
\usepackage{xcolor}
\theoremstyle{plain}
\newtheorem{thm}{Theorem}[subsection]
\newtheorem{lem}{Lemma}[subsection]

\newtheorem{cor}{Corollary}[subsection]

\newtheorem*{prob}{Problem}
\theoremstyle{definition}
\newtheorem{dfn}{Definition}[subsection]
\newtheorem{ex}{Example}[subsection]
\newtheorem*{theorem}{Main Theorem}
\newtheorem*{acknowledgements}{Acknowledgement}

\begin{document}

\setlength{\baselineskip}{12pt}

\title{Kirby-Thompson invariants of distant sums of standard surfaces}
\author{Minami Taniguchi}

\keywords{Knotted surface, Bridge trisection, Kirby-Thompson invariant}

\subjclass[2020]{Primary 57K10; Secondary 57K20}

\maketitle

%アブストラクト

\begin{abstract}
Blair--Campisi--Taylor--Tomova \cite{Kirby_Thompson_1} defined the $\mathcal{L}$-invariant $\mathcal{L}(F)$ of a knotted surface $F$, using pants complexes of trisection surfaces of bridge trisections of $F$. After that, Aranda--Pongtanapaisan--Zhang \cite{Kirby_Thompson_3} introduced the $\mathcal{L}^{*}$-invariant $\mathcal{L}^{*}(F)$ using dual curve complexes instead of pants complexes. In this paper, we determine both of $\mathcal{L}$-invariant and $\mathcal{L}^{*}$-invariant of any finite distant sum of standard surfaces, and this is the first example of knotted surfaces whose bridge numbers and these invariants can be arbitrary large.
\end{abstract}

%イントロダクション

\section{Introduction}

Inspired by a trisection of a smooth oriented connected closed 4-manifold first defined by Gay--Kirby \cite{trisection_1}, Meier--Zupan \cite{bridge_trisection_1} defined a bridge trisection of a knotted surface in $S^{4}$, which is an analogue of bridge decompositions of classical links in $S^{3}$. Roughly speaking, a bridge trisection of a knotted surface in $S^{4}$ is a splitting of it into three simple pieces, called trivial disk systems. Meier and Zupan showed that every knotted surface in $S^{4}$ has a bridge trisection and any two bridge trisections of the same surface are related by a finite sequence of (de)stabilizations. Using a notion of bridge trisections, we obtain a new way of depicting a given bridge trisected surface, which is called a tri-plane diagram consisting of three trivial tangle diagrams each of which has the same bridge number. 

After that, Blair--Campisi--Taylor--Tomova \cite{Kirby_Thompson_1} defined the \textbf{$\mathcal{L}$-invariant} of a knotted surface in $S^{4}$ using minimal bridge trisections of the surface and pants complexes of the trisection surfaces of the bridge trisections. The invariant measures a complexity of the given knotted surface. They showed that, if the $\mathcal{L}$-invariant of a knotted surface is equal to zero, then it is smoothly isotopic to a finite distant sum of conneced sum of finitely many standard surfaces except standard tori. 

Aranda--Pongtanapaisan--Zhang \cite{Kirby_Thompson_3} defined a new invariant, called the \textbf{$\mathcal{L}^{*}$-invariant} using dual curve complexes instead of pants complexes. They showed that a knotted surface whose $\mathcal{L}^{*}$-invariant is less than or equal to two becomes smoothly isotopic to a finite distant sum of conneced sum of finitely many standard surfaces except standard tori. Hence, the $\mathcal{L}$-invariant and the $\mathcal{L}^{*}$-invariant give us a sufficient condition of a knotted surface to be an unlink in $S^{4}$.

It is rather difficult to calculate the $\mathcal{L}$-invariant and the $\mathcal{L}^{*}$-invariant with large bridge number, because these invariants are defined as the minimum values over special pants distances with respect to three trivial tangles forming the spine of a bridge trisection. In fact, values of the $\mathcal{L}$-invariant (or the $\mathcal{L}^{*}$-invariant) have been determined only for knotted surfaces with bridge number less than or equal to six (see \cite{Kirby_Thompson_1}, \cite{Kirby_Thompson_2}, and \cite{Kirby_Thompson_3} for more details). In this paper, we determine these invariants of distant sums of finite standard unknotted surfaces allowing the standard tori, and this is the first example of knotted surfaces whose bridge numbers and the values of these invariants can be arbitrary large.

\begin{theorem}[Theorem \ref{thm:main theorem}]
If $F$ is a finite distant sum of standard unknotted surfaces in $S^{4}$ and $n(F)$ is the number of standard unknotted tori included in $F$, then $\mathcal{L}(F)=\mathcal{L}^{*}(F)=3n(F).$
\end{theorem}

This paper is organized as follows. In Section 2, we review the definitions of bridge trisections of knotted surfaces, pants complexes, and dual curve complexes. In Section 3, we recall the definitions of the $\mathcal{L}$-invariant and the $\mathcal{L}^{*}$-invariant of knotted surfaces and see some simple examples. Finally, we prove the main result in Section 4. Note that every manifold in this paper is supposed to be smooth.

\begin{acknowledgements}

The author would like to express sincere gratitude to his supervisor, Hisaaki Endo, for his support and encouragement throughout this work.

\end{acknowledgements}

%準備の章

\section{Preliminaries}

In order to define the $\mathcal{L}$-invariant and the $\mathcal{L}^{*}$ invariant of a knotted surface in $S^{4}$, we must recall bridge trisections of the knotted surface, pants complexes, and dual curve complexes of the bridge surfaces of the trisection.

%ブリッジトライセクションの節

\subsection{Bridge trisection}

\quad\, In this subsection, we recall the definition of a bridge trisection of a knotted surface in $S^{4}$ and some concepts around it in order to discuss the main theorem. Let $F\subset{S^{4}}$ be a knotted surface. Note that, in this paper, a pair $(S^{4},F)$ is also refered to as a knotted surface.
\begin{dfn}
Let $b,c_{1},c_{2},c_{3}$ be positive integers. A \textit{$(b;c_{1},c_{2},c_{3})$-bridge trisection} of $(S^{4},F)$ is a decomposition $(X_{1},\mathcal{D}_{1})\cup(X_{2},\mathcal{D}_{2})\cup(X_{3},\mathcal{D}_{3})$ satisfying the following properties, where $b=|\mathcal{D}_{1}\cap\mathcal{D}_{2}\cap\mathcal{D}_{3}|/2$ and $c_{i}=|\mathcal{D}_{i}|$.
\begin{itemize}
\item[(i)] $S^{4}=X_{1}\cup X_{2}\cup X_{3}$ is the genus 0 trisection of $S^{4}$.
\item[(ii)] For each $i\in{\{1,2,3\}}$, $(X_{i},\mathcal{D}_{i})$ is a trivial $c_{i}$-disk system.
\item[(iii)] For each $i\not=j\in{\{1,2,3\}}$, $(B_{ij},\alpha_{ij}):=(X_{i},\mathcal{D}_{i})\cap(X_{j},\mathcal{D}_{j})=(\partial X_{i},\partial\mathcal{D}_{i})\cap(\partial X_{j},\partial\mathcal{D}_{j})$ is a trivial $b$-tangle.
\item[(iv)] For each $\{i,j,k\}=\{1,2,3\}$, $(\partial X_{i},\partial\mathcal{D}_{i})=(B_{ij},\alpha_{ij})\cup_{\partial}\overline{(B_{ki},\alpha_{ki})}$ is a $b$-bridge splitting of the unlink $\partial\mathcal{D}_{i}$.
\end{itemize}
\end{dfn}
From the above properties, it follows that $(\Sigma,\mathrm{p}):=(B_{ij},\alpha_{ij})\cap(B_{ki},\alpha_{ki})$ is a $2b$ punctured 2-sphere, which is called a \textit{bridge surface} of the given bridge trisection, and each $(\partial X_{i},\partial\mathcal{D}_{i})$ is the unlink whose number of link components is $c_{i}$. Let $\mathcal{T}$ denote the bridge trisection and the union $\bigcup_{i\not=j}(B_{ij},\alpha_{ij})$ is called the \textit{spine} of $\mathcal{T}$. We often refer to $(b;c,c,c)$-bridge trisections as $(b;c)$-bridge trisections. An integer $b$ is called the \textit{bridge number} of the given bridge trisection $\mathcal{T}$ and the minimum number of such integers $b$ over all bridge trisections of $F$ is called the \textit{bridge number} of the surface $F$. Since a way of attatching 2-disks in $B^{4}$ along an unlink in $S^{3}=\partial B^{4}$ is unique up to isotopy relative to the boundary, the equivalent class of $\mathcal{T}$ is completely determined by the spine of $\mathcal{T}$. Figure \ref{fig:bridge trisection of the unknotted sphere} below shows an image of the $(2;1)$-bridge trisection of the unknotted 2-sphere $\mathcal{U}$.

%ブリッジトライセクションのイメージ図

\begin{figure}[h]
\center
\includegraphics[width=35mm]{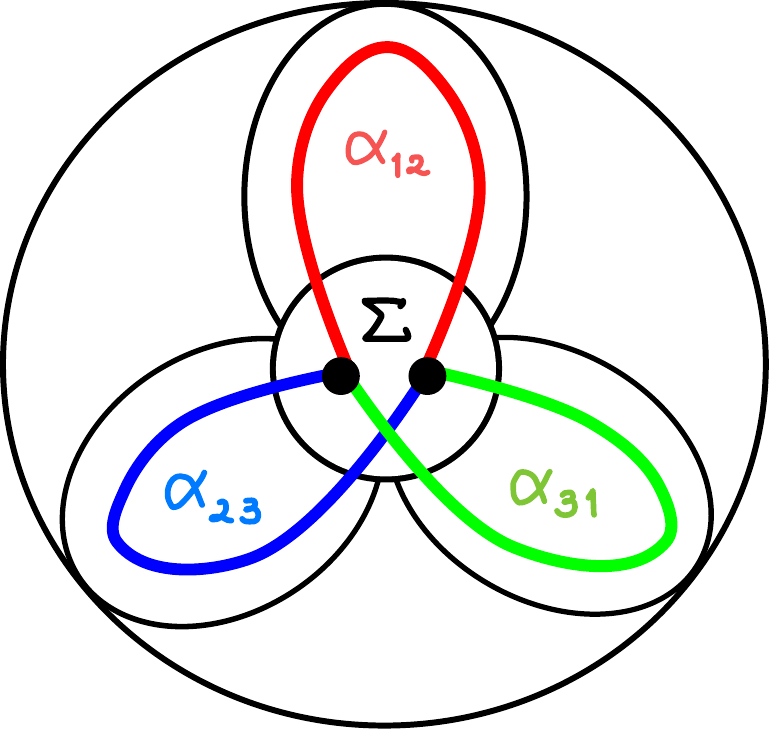}
\caption{An image of the $(2;1)$-bridge trisection of $\mathcal{U}$, where the three colored arcs represent its spine.}
\label{fig:bridge trisection of the unknotted sphere}
\end{figure}

Two bridge trisections $\mathcal{T}$ and $\mathcal{T}'$ of $(S^{4},F)$ are said to be \textit{equivalent} if there is a smooth isotopy of $(S^{4},F)$ carrying the components of $\mathcal{T}$ to the corresponding components of $\mathcal{T}'$ and, we write $\mathcal{T}\simeq\mathcal{T}'$. Meier and Zupan \cite{bridge_trisection_1} showed that any two bridge trisecrions of $F$ become equivalent by a finite stabilization and destabilization operations and any two tri-plane diagrams of them are related by a finite tri-plane moves. If we stabilize a bridge trisection $\mathcal{T}$ once, then the bridge number of $\mathcal{T}$ increases by one. A bridge trisection $\mathcal{T}$ is said to be \textit{minimal} if $\mathcal{T}$ is unstabilized and, in this situation, the bridge number of $\mathcal{T}$ is same as that of $F$. In this paper, since we are interested in the minimal bridge trisections of the standard surfaces, we do not refer to detailed definitions of (de)stabilization operations. See \cite{bridge_trisection_1} for a more detailed explanation of (de)stabilization.

Using a bridge trisection $\mathcal{T}$ of a given knotted surface in $S^{4}$, we obtain a new way of depicting the surface, which is called a \textit{tri-plane diagram} representing the isotopy class of the spine of $\mathcal{T}$. For a detail way of obtaining tri-plane diagrams, see \cite{bridge_trisection_1}, and see \cite{banded_unlink_diagram} for banded unlink diagrams. Figure \ref{fig:triplane diagrams of standard surfaces} shows some examples of tri-plane diagrams of seven simple surfaces $\{\mathcal{U},\mathcal{P}_{+},\mathcal{P}_{-},\mathcal{K}_{2,0},\mathcal{K}_{1,1},\mathcal{K}_{0,2},\mathcal{T}\}$, and they and the bridge trisections of them up to stabilizations and destabilizations are called \textit{standard}, where $\mathcal{U}$ is the \textit{unknotted 2-sphere}, $\mathcal{P}_{\pm}$ is the \textit{unknotted projective planes} with the normal Euler number $\pm2$, respectively, $\mathcal{K}_{ij}:=(\#^{i}\mathcal{P}_{+})\#(\#^{j}\mathcal{P}_{-})$ is the \textit{unknotted Klein bottles}, and $\mathcal{T}$ is the \textit{unknotted torus}. From \cite{bridge_trisection_1}, all of the tri-plane diagrams in Figure \ref{fig:triplane diagrams of standard surfaces} are unstabilized.

%標準曲面のトリプレーンダイアグラムの図

\begin{figure}[h]
\center
\includegraphics[width=65mm]{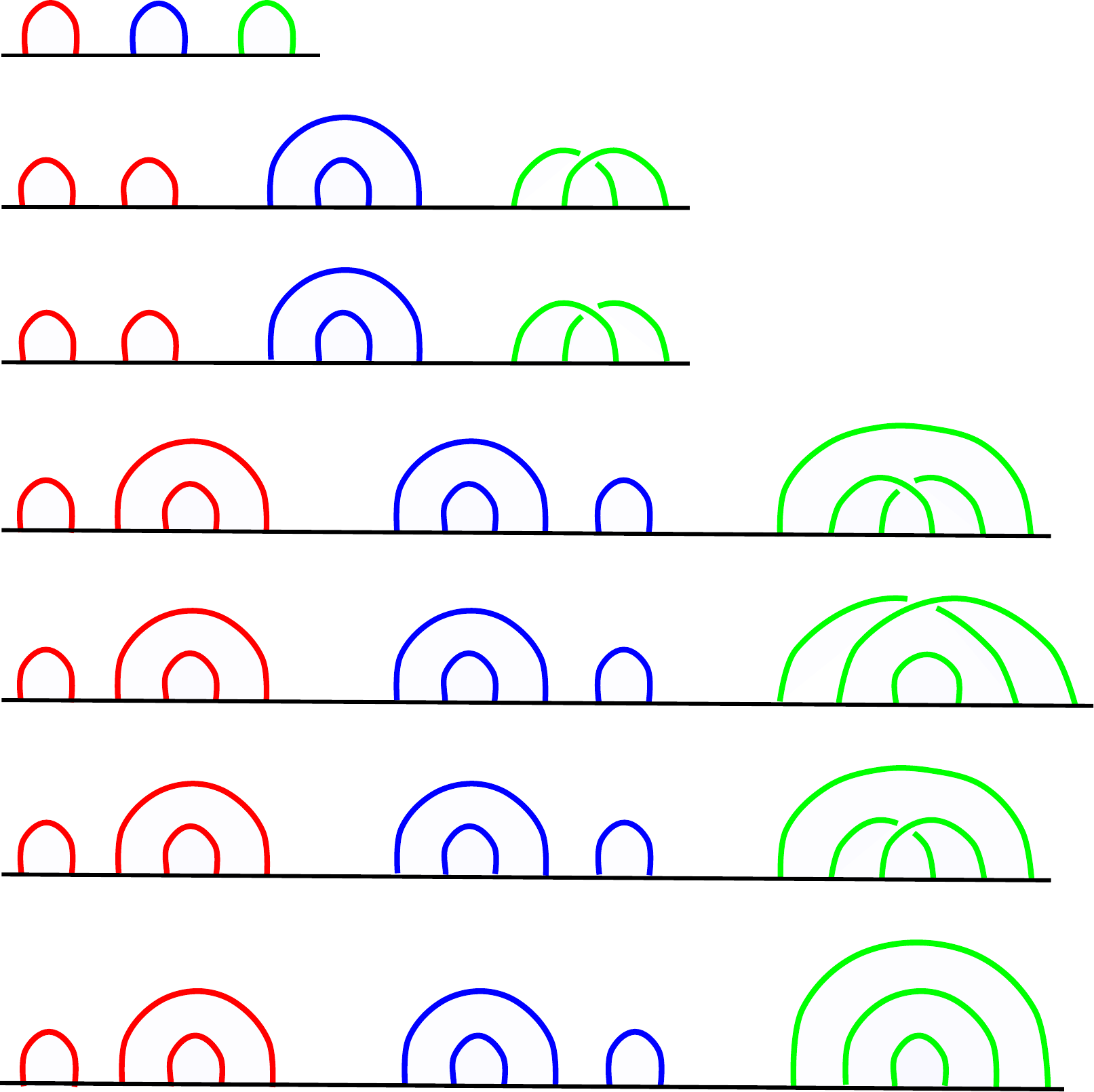}
\caption{Tri-plane diagrams of the standard unknotted surfaces, from the top to the bottom, representing $\mathcal{U},\mathcal{P}_{+},\mathcal{P}_{-},\mathcal{K}_{2,0},\mathcal{K}_{1,1},\mathcal{K}_{0,2}$, and $\mathcal{T}$, respectively.}
\label{fig:triplane diagrams of standard surfaces}
\end{figure}

Next, we introduce a way of obtaining new bridge trisections and new knotted surfaces from bridge trisections of finite knotted surfaces. Let $(S^{4},F_{1})$ and $(S^{4},F_{2})$ be knotted surfaces and $\mathcal{T}_{1}$ and $\mathcal{T}_{2}$ be bridge trisections of $F_{1}$ and $F_{2}$, respectively. Let $\Sigma_{1}$ and $\Sigma_{2}$ be the bridge surfaces of $\mathcal{T}_{1}$ and $\mathcal{T}_{2}$, respectively. Let $a_{*}\in{\Sigma_{*}}$ be a point on  the bridge surface and $\nu(a_{*})\subset{S^{4}}$ be a quite small neighborhood of $a_{*}$ in $S^{4}$ satisfying  that $\nu(a_{*})\cap F_{*}$ is either the empty set or a unit circle $S^{1}$ for each $*\in{\{1,2\}}$. We obtain a new knotted surface and a new bridge trisection of it by pasting $(\overline{S^{4}-\nu(a_{1})},\overline{F_{1}-\nu(a_{1})})$ and $(\overline{S^{4}-\nu(a_{2})},\overline{F_{2}-\nu(a_{2})})$ along their boundaries and there exist two distinct situations. If $a_{*}\notin{F_{*}\cap\Sigma_{*}}$, then we obtain the new bridge trisection of the distant sum $F_{1}\sqcup F_{2}$ and $\mathcal{T}_{1}\sqcup\mathcal{T}_{2}$ denotes the new bridge trisection of $F_{1}\sqcup F_{2}$. We call it the \textit{distant sum} of $\mathcal{T}_{1}$ and $\mathcal{T}_{2}$. On the other hand, if $a_{*}\in{F_{*}\cap\Sigma_{*}}$, then we obtain the new bridge trisection of the connected sum $F_{1}\#F_{2}$ and $\mathcal{T}_{1}\#\mathcal{T}_{2}$ denotes the new bridge trisecion of $F_{1}\#F_{2}$. We call it the \textit{connected sum} of $\mathcal{T}_{1}$ and $\mathcal{T}_{2}$.

From the following lemma, we know that an arbitrary unstabilized bridge trisection of a distant sum of some knotted surfaces is equivalent to the distant sum of unstabilized bridge trisections of the knotted surfaces.

\begin{lem}
Let $F$ be the distant sum of knotted surfaces $F_{1},\cdots,F_{n}$ and $\mathcal{T}$ an unstabilized bridge trisection of $F$. Then, $\mathcal{T}$ is equivalent to a distant sum $\mathcal{T}_{1}\sqcup\cdots\sqcup\mathcal{T}_{n}$, where each $\mathcal{T}_{i}$ is an unstabilized bridge trisection of $F_{i}$.
\label{lem:bridge trisection of distant sum}
\end{lem}

\begin{proof}
Let $\mathcal{T}$ be an unstabilized bridge trisection of $F$. Since $F$ is the distant sum of $F_{1},\cdots,F_{n}$, there exist $n-1$ decomposing 3-spheres $Q_{1},\cdots,Q_{n-1}\subset{S^{4}}$ such that each $Q_{i}$ separates $F_{i}$ from the others. We surger $(S^{4},F)$ along each $Q_{i}$ to obtain $n$ knotted surfaces $(S^{4},F_{i})$, and it inherits the bridge trisection $\mathcal{T}_{i}$ from $\mathcal{T}$. It is easy to see that each $\mathcal{T}_{i}$ is unstabilized since $\mathcal{T}$ is unstabilized. Furthermore, we can construct the unstabilized bridge trisection $\mathcal{T}_{1}\sqcup\cdots\sqcup\mathcal{T}_{n}$ and it is equivalent to the original bridge trisection $\mathcal{T}$.
\end{proof}

Let $\mathcal{T}$ be a bridge trisection of a knotted surface $(S^{4},F)$. $\mathcal{T}$ is said to be \textit{reducible} if there exist two bridge trisections $\mathcal{T}_{1}$ and $\mathcal{T}_{2}$ such that either $\mathcal{T}\simeq\mathcal{T}_{1}\#\mathcal{T}_{2}$ or $\mathcal{T}\simeq\mathcal{T}_{1}\sqcup\mathcal{T}_{2}$ holds, where both $\mathcal{T}_{1}$ and $\mathcal{T}_{2}$ are not equivalent to the $(2;1)$-bridge trisection of $\mathcal{U}$ when $\mathcal{T}\simeq\mathcal{T}_{1}\#\mathcal{T}_{2}$. $\mathcal{T}$ is said to be \textit{irreducible} if $\mathcal{T}$ is not reducible.

The following lemma implies a necessary and sufficient condition of bridge trisections to be reducible or stabilized. You see the definition of c-reducing curves in the next subsection.

\begin{lem}[Lemma 2.10 of \cite{Kirby_Thompson_2}]
Let $\mathcal{T}$ be a bridge trisection of a knotted surface whose spine is $\bigcup_{i\not=j}(B_{ij},\alpha_{ij})$. Then, $\mathcal{T}$ is reducible or stabilized if and only if there exists a c-reducing curve running on the bridge surface $\Sigma$ for each tangle $\alpha_{ij}$.
\label{lem:reducible and stabilized}
\end{lem}

\subsection{Pants complex and dual curve complex}

\quad\, In this subsection, we recall the pants complex $P(\Sigma)$ and the dual curve complex $C^{*}(\Sigma)$ of a punctured 2-sphere $\Sigma=\Sigma_{0,n}$. The $n$ punctured sphere $\Sigma$ is said to be \textit{admissible} if $n\geq4$ holds. Assume that $\Sigma$ is admissible and $c\subset{\Sigma}$ is a simple closed curve on $\Sigma$. The curve $c\subset{\Sigma}$ is said to be \textit{essential} if $c\cap\mathrm{p}=\phi$ and $c$ bounds neither an unpunctured 2-disk in $\Sigma$ nor a once punctured 2-disk in $\Sigma$, where $\mathrm{p}\subset{\Sigma}$ denotes the set of $n$ punctures. 
A compact surface with three boundary components, at most two of which are allowed to be punctures on it, is said to be a \textit{pair of pants} and an image of it is shown in Figure \ref{fig:pair of pants}. A family of essential simple closed curves $\{c_{l}\}_{l\in{L}}$ on $\Sigma$ is said to be a \textit{pants decomposition} of $\Sigma$ if we obtain pairs of pants by cutting $\Sigma$ along all curves $c_{l}$. We refer them up to isotopy relative to punctures on $\Sigma$. It is easy to see that any pants decomposition of $\Sigma=\Sigma_{0,n}$ consists of $n-3$ essential curves.

%パンツの図

\begin{figure}[t]
\centering
\includegraphics[width=50mm]{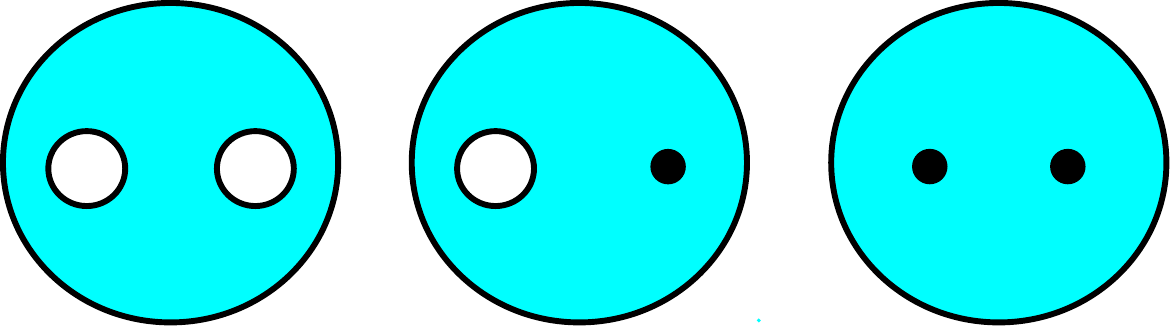}
\caption{An image of a pair of pants.}
\label{fig:pair of pants}
\end{figure}

The \textit{pants complex} $P(\Sigma)$ of an $n$ punctured 2-sphere $\Sigma$ is a 1-complex defined as follows. A \textit{vertex} of $P(\Sigma)$ is a pants decomposition of $\Sigma$. Two vertices $c=\{c_{l}\}_{l},c'=\{c'_{l}\}_{l}$ of $P(\Sigma)$ are connected by a length one \textit{edge} in $P(\Sigma)$ if it follows that $c_{l}=c_{l}'$ for each $l\in{\{1,\cdots,n-4\}}$ and $|c_{n-3}\cap c_{n-3}'|=2$. If vertices $c$ and $c'$ are connected by an edge in $P(\Sigma)$, we say that $c$ is obtained from $c'$ by an \textit{A-move} and let denote this situation as $c\rightarrow c'$. Let $c,c'\in{P(\Sigma)}$ be two vertices. Then, the \textit{distance} $d(c,c')$ between $c$ and $c'$ in $P(\Sigma)$ is the number of length one edges in a shortest path connecting them.

Similar to the pants complex of $\Sigma$, we can define the \textit{dual curve complex} $C^{*}(\Sigma)$ as follows. A \textit{vertex} of $C^{*}(\Sigma)$ is a pants decomposition of $\Sigma$. Two vertices $c=\{c_{l}\}_{l},c'=\{c_{l}'\}_{l}$ of $C^{*}(\Sigma)$ are connected by a length one \textit{edge} in $C^{*}(\Sigma)$ if it follows that $c_{l}=c_{l}'$ for each $l\in{\{1,\cdots,n-4\}}$ and $|c_{n-3}\cap c_{n-3}'|\geq2$. As in the definition of the pants complex of $\Sigma$, $d^{*}(c,c')$ denotes the \textit{distance} between $c$ and $c'$ in $C^{*}(\Sigma)$. We see easily that the 0-skeleton of $P(\Sigma)$ and of $C^{*}(\Sigma)$ are same and the pants complex $P(\Sigma)$ is a subcomplex of the dual curve complex $C^{*}(\Sigma)$ since any two vertices $c,c'\in{P(\Sigma)}$ connected by an edge in $P(\Sigma)$ are also connected by the same edge in $C^{*}(\Sigma)$, following the inequation $d^{*}(c,c')\leq d(c,c')$ for any vertices $c,c'\in{P(\Sigma)}$.

Let $(B,\alpha)$ be a trivial $b$-tangle with $b\geq2$ and $\Sigma$ denotes the $2b$ punctured 2-sphere appearing as the boundary of the given tangle, where we consider $\partial\alpha\subset{\Sigma}$ as the punctures on $\Sigma$. An essential simple closed curve $c\subset{\Sigma}$ is said to be a \textit{compressing curve} if there exists a 2-disk $D_{c}$ properly embedded in $B$ such that $\partial D_{c}=c$ and $D_{c}\cap\alpha=\phi$, and such a 2-disk $D_{c}$ is called a \textit{compressing disk} for $\alpha$. On the other hand, an essential simple closed curve $c\subset{\Sigma}$ is said to be a \textit{cut curve} if there exists a 2-disk $D_{c}$ properly embedded in $B$ such that $\partial D_{c}=c$ and $|D_{c}\cap\alpha|=1$, and such a 2-disk is called a \textit{cut disk} for $\alpha$. It follows easily that any compressing curve bounds an even number of punctures of $\Sigma$ and any cut curve bounds an odd number of punctures of $\Sigma$.

%c-縮小曲線の図とエフィシェントなパンツ分解の図

\begin{figure}[t]
\centering
\begin{minipage}[b]{0.49\columnwidth}
\centering
\includegraphics[width=0.9\columnwidth]{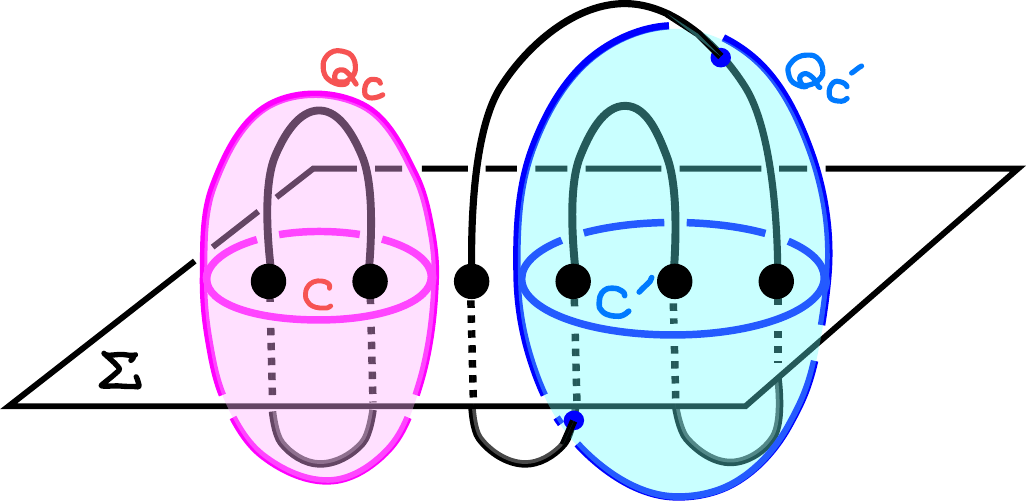}
\caption{An image of c--reducing spheres for the 3--bridge unlink.}
\label{fig:c-spheres}
\end{minipage}
\begin{minipage}[b]{0.49\columnwidth}
\centering
\includegraphics[width=0.9\columnwidth]{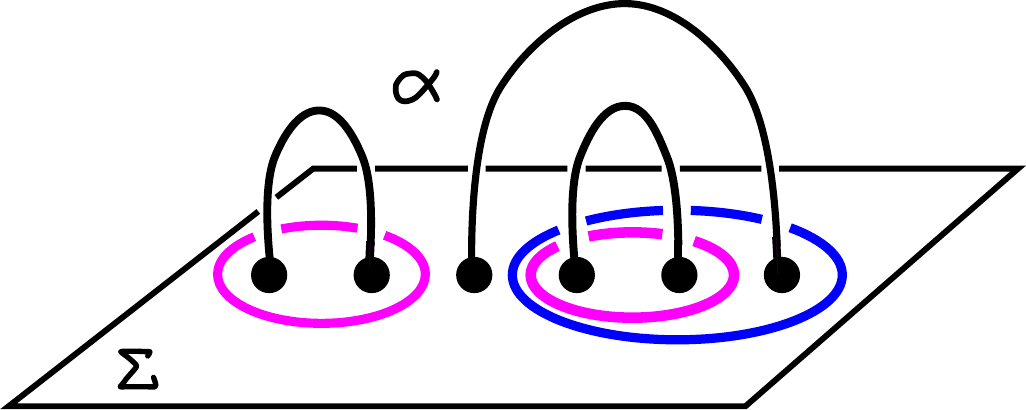}
\caption{An example of an efficient pants decomposition for the trivial 3-tangle.}
\label{fig:efficient pants decomposition}
\end{minipage}
\end{figure}

Let $(S^{3},L)=(B_{+},\alpha_{+})\cup_{\partial}(B_{-},\alpha_{-})$ be a bridge split unlink whose bridge surface is $\Sigma$ and suppose $|L\cap\Sigma|\geq4$. An essential simple closed curve $c\subset{\Sigma}$ is said to be \textit{reducing} if there exists a 2-sphere $Q_{c}$ embedded in $S^{3}$ such that $Q_{c}\cap\Sigma=c$ and $D_{c}^{+}\cap\alpha_{+}=D_{c}^{-}\cap\alpha_{-}=\phi$, where $D_{c}^{\pm}:=Q_{c}\cap B_{\pm}$, and such a 2-sphere is called a \textit{reducing sphere} for the unlink $L$. On the other hand, an essential simple closed curve $c\subset{\Sigma}$ is said to be \textit{cut-reducing} if there exists a 2-sphere $Q_{c}$ embedded in $S^{3}$ such that $Q_{c}\cap\Sigma=c$ and $|D_{c}^{+}\cap\alpha_{+}|=|D_{c}^{-}\cap\alpha_{-}|=1$, where $D_{c}^{\pm}=Q_{c}\cap B_{\pm}$, and such a 2-sphere is called a \textit{cut-reducing sphere} for the unlink $L$ (See Figure \ref{fig:c-spheres}).

Let $(B,\alpha)$ be a trivial $b$-tangle and $c=\{c_{l}\}_{l}$ be a pants decomposition of $2b$ punctured 2-sphere $\Sigma$, where $b\geq4$. If each $c_{l}\in{c}$ is either a compressing curve or a cut curve for $\alpha$, then $c$ is called \textit{efficient} for the tangle $\alpha$ (see Figure \ref{fig:efficient pants decomposition}). It is easy to see that there exist efficient pants decompositions for any trivial $b$-tangle $(B,\alpha)$ with $b\geq2$.

For simplicity, we introduce the following words. Let $\Sigma$ be an $n$ punctured 2-sphere and $c$ be a pants decomposition of $\Sigma$, where $n\geq4$. We consider the surface $\Sigma$ as $\{(x_{1},x_{2},x_{3})\in{\mathbb{R}^{3}}\,|\,x_{1}^{2}+x_{2}^{2}+x_{3}^{2}=1\}$ and $D^{2}_{+}\subset{\Sigma}$ denotes the upper half disk $\{(x_{1},x_{2},x_{3})\in{\Sigma}\,|\,x_{3}\geq0\}$ of $\Sigma$. Let all punctures on $\Sigma$ be inside $D_{+}^{2}$ and, by an appropriate homeomorphism on $\Sigma$ relative to the punctures, assume that all curves of $c$ are inside $D_{+}^{2}$. Let $c_{l}\in{c}$. $S_{c_{l}}^{\text{in}}$ and $S_{c_{l}}^{\text{out}}$ denote the subsurfaces in $D_{+}^{2}$ such that $S_{c_{l}}^{\text{in}}\cup S_{c_{l}}^{\text{out}}=D_{+}^{2}, \partial S_{c_{l}}^{\text{in}}=c_{l}$, and $\partial S_{c_{l}}^{\text{out}}=c_{l}\cup\partial D_{+}^{2}$. Let $A\subset{D_{+}^{2}}$ be a subset satisfying that either $A\subset{\text{int}S_{c_{l}}^{\text{in}}}$ or $A\subset{\text{int}S_{c_{l}}^{\text
{out}}}$ holds. The subset $A$ is \textit{inside} $c_{l}$ or \textit{bounded} by $c_{l}$ if $A\subset{\text{int}S_{c_{l}}^{\text{in}}}$ holds, and if not, $A$ is said to be \textit{outside} $c_{l}$. From the definition of the pants decomposition of $\Sigma$, for each essential simple closed curve $c_{l}\in{c}$, there exist two pairs of pants intersecting along the curve $c_{l}$ such that the interior of one of them is inside $c_{l}$ and that of the other is outside $c_{l}$. Let $\pi_{c_{l}}^{\text{in}}\subset{\Sigma}$ denote the pair of pants such that $c_{l}\subset{\partial\pi_{c_{l}}^{\text{in}}}$ and its interior region is inside $c_{l}$, and let $\pi_{c_{l}}^{\text{out}}\subset{\Sigma}$ denote the pair of pants such that $c_{l}\subset{\partial\pi_{c_{l}}^{\text{in}}}$ and its interior region is outside $c_{l}$. You see images of $\pi_{c_{l}}^{\text{in}}$ and $\pi_{c_{l}}^{\text{out}}$ in Figure \ref{fig:inside and outside of pants}, where the outmost curve represents $\partial D_{+}^{2}$.

%エッセンシャル曲線の内部と外部の図

\begin{figure}[h]
\center
\includegraphics[width=60mm]{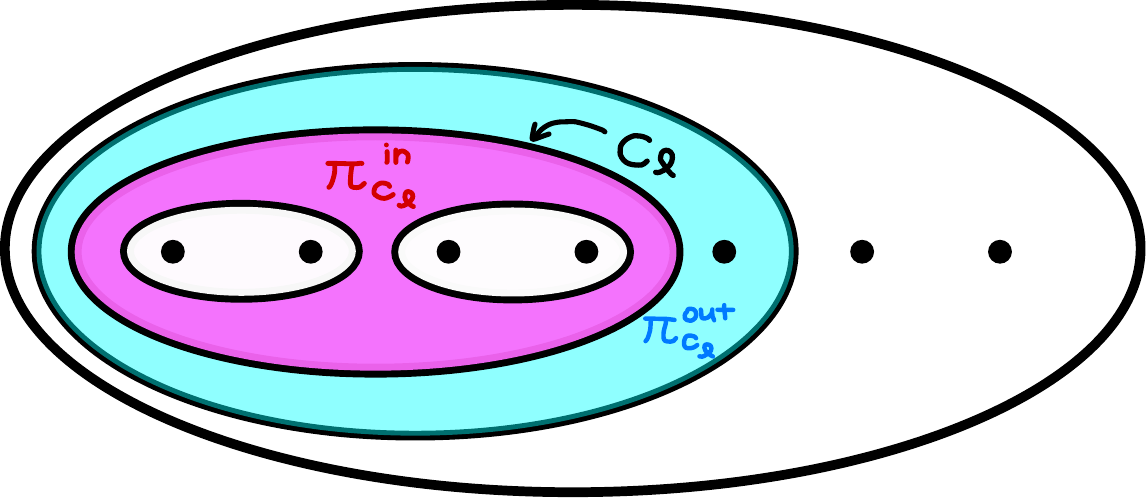}
\caption{Examples of $\pi_{c_{l}}^{\text{in}}$ (the pink region) and $\pi_{c_{l}}^{\text{out}}$ (the blue region).}
\label{fig:inside and outside of pants}
\end{figure}

\section{Kirby-Thompson invariant}

In this section, we review the definition of the $\mathcal{L}$-invariant and the $\mathcal{L}^{*}$-invariant of knotted surfaces and show some simple examples. In Subsection 3.1, we need to recall some specific pants decompositions with respect to given bridge trisections, and introduce these invariants in Subsection 3.2.

\subsection{Some special pants decompositions}

In order to define the $\mathcal{L}$-invariant and the $\mathcal{L}^{*}-$invariant of a knotted surface, we have to define a special pair of pants decompositions of the bridge surfaces of bridge trisections of the surface, which plays a key role when defining or discussing these invariants.

Let $(S^{3},L)=(B_{+},\alpha_{+})\cup(B_{-},\alpha_{-})$ be a $b$-bridge split unlink, and $\Sigma$ denotes the bridge surface of the splitting, considering it as a $2b$ punctured sphere with $b\geq2$. A pants decomposition $p_{+}\in{P(\Sigma)}$ is said to be \textit{efficient} for the tangle $\alpha_{+}$ if every essential simple closed curve $c_{l}\in{p_{+}}$ is either compressing or cut for $\alpha_{+}$, and $P_{c}(\alpha_{+})\subset{P(\Sigma)}$ denotes the set of efficient pants decompositions for the tangle $\alpha_{+}$. A pair of pants decompositions $(p_{+},p_{-})\in{P_{c}(\alpha_{+})\times P_{c}(\alpha_{-})}$ is called an \textit{efficient defining pair} for the unlink $L$ with the bridge surface $\Sigma$ if $d(p_{+},p_{-})=d(P_{c}(\alpha_{+}),P_{c}(\alpha_{-}))$ holds, that is, an efficient defining pair of $L$ is a pair of efficient pants decompositions for the upper tangle and the lower tangle whose pants distance in $P(\Sigma)$ is minimal. The following lemma shows some important features of efficient defining pairs of bridge split unlink whose bridge number is greater than or equal to two.

\begin{lem}[Lemma 3.4 of \cite{Kirby_Thompson_3}]
Let $(p_{+},p_{-})\in{P_{c}(\alpha_{+})\times P_{c}(\alpha_{-})}$ be an efficient defining pair for a $b$-bridge split unlink $(S^{3},L)=(B_{+},\alpha_{+})\cup(B_{-},\alpha_{-})$ with $b\geq2$. Let $c=|L|$. We can decompose each pants decomposition as $p_{+}=\psi\cup g_{+}$ and $p_{-}=\psi\cup g_{-}$ so that the following properties hold.
\begin{itemize}
\item[(i)] $|g_{+}|=|g_{-}|=b-c$ and $|\psi|=b+c-3$.
\item[(ii)] For any shortest path $p_{+}\rightarrow p_{-}$ realizing the distance $d(p_{+},p_{-})$, every curve in $g_{+}$ is moved to some essential curve in $g_{-}$ by an A-move.
\item[(iii)] For any curves $c\in{g_{+}}$ and $c'\in{g_{-}}$, $|c\cap c'|\leq2$ holds.
\item[(iv)] For any curve $c\in{g_{+}}$, there exists $c'\in{g_{-}}$ uniquely such that $|c\cap c'|=2$.
\end{itemize}
\label{lem:efficient defining pair 1}
\end{lem}

Let $(S^{4},F)$ be a knotted surface and $\mathcal{T}$ be an unstabilized $(b;c_{1},c_{2},c_{3})$-bridge trisection of $F$. Let $\bigcup_{i\not=j}(B_{ij},\alpha_{ij})$ be the spine of $\mathcal{T}$ and $\Sigma$ the bridge surface of $\mathcal{T}$. As in Subsection 2.2, we consider the surface $\Sigma$ as a $2b$ punctured 2-sphere. We note that $d^{*}(p_{ij}^{i},p_{ki}^{i})=b-c_{i}$ holds.

\begin{dfn}[$\mathcal{L}$-invariant]
We define the \textit{$\mathcal{L}$-invariant} $\mathcal{L}(\mathcal{T})$ of the bridge trisection $\mathcal{T}$ as follows. If $\mathcal{T}$ is the $(2;1)$-bridge trisection of $\mathcal{U}$, then we define $\mathcal{L}(\mathcal{T}):=0$. If not, we define
\[
\mathcal{L}(\mathcal{T}):=\text{min}\left\{\sum_{i\not=j}d(p_{ij}^{i},p_{ij}^{j})\middle|(p_{ij}^{i},p_{ki}^{i})\text{ is an efficient defining pair for }\partial\mathcal{D}_{i}.\right\}.
\]
The \textit{$\mathcal{L}$-invariant} $\mathcal{L}(F)$ of the surface $F$ is defined as the minimum value of $\mathcal{L}(\mathcal{T})$ over all unstabilized bridge trisections $\mathcal{T}$ of $F$. If we use the distance function $d^{*}$ instead of $d$ in the definition of the invariant $\mathcal{T}$, we obtain the other invariant $\mathcal{L}^{*}(\mathcal{T})$ and $\mathcal{L}^{*}(F)$ by the same way, and we call them the \textit{$\mathcal{L}^{*}$-invariant}. Note that we also call these invariants the \textit{Kirby-Thompson invariants}. It is easy to see that $\mathcal{L}\geq\mathcal{L}^{*}\geq0$.

\begin{figure}[ht]
\center
\includegraphics[width=55mm]{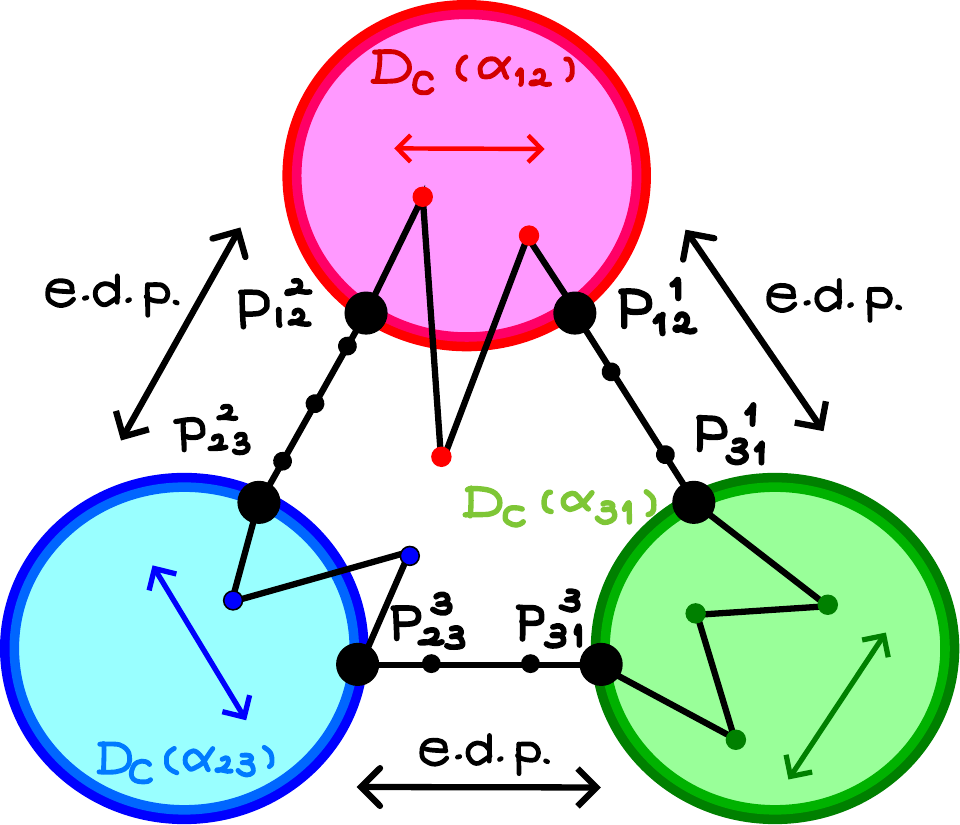}
\caption{An image of the Kirby-Thompson invariant.}
\end{figure}

\end{dfn}

%Kirby-Thompson不変量のイメージ図

\begin{ex}
We can calculate upper bounds for the $\mathcal{L}$-invariant and the $\mathcal{L}^{*}$-invariant of the standard unknotted surfaces. Figure \ref{fig:efficient defining pairs of standard surfaces} shows that $\mathcal{L}(F)\leq0, \mathcal{L}^{*}(F)\leq0$, hence $\mathcal{L}(F)=\mathcal{L}^{*}(F)=0$ holds for any standard surfaces $F\not=\mathcal{T}$. On the other hand, from the bottom figure in Figure \ref{fig:efficient defining pairs of standard surfaces}, we see $\mathcal{L}(\mathcal{T})\leq3,\mathcal{L}^{*}(\mathcal{T})\leq3$. Moreover, from Example 5.4 of \cite{Kirby_Thompson_3}, $\mathcal{L}(\mathcal{T})=\mathcal{L}^{*}(\mathcal{T})=3$ holds.\\

\begin{figure}[h]
\center
\includegraphics[width=77mm]{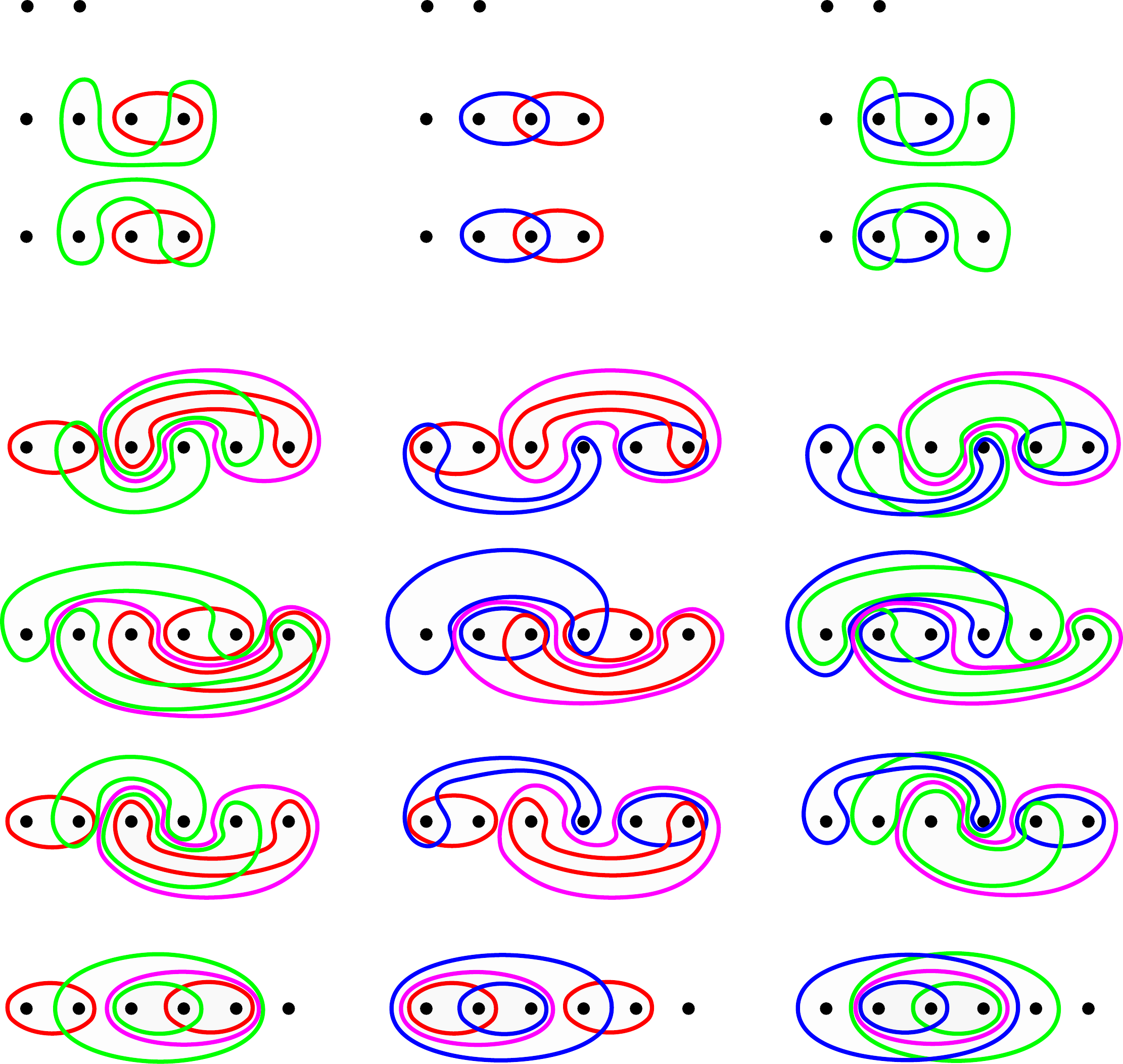}
\caption{Efficient defining pairs for the unstabilized bridge trisections of the standard surfaces in Figure \ref{fig:triplane diagrams of standard surfaces}, from the top to the bottom, representing efficient pairs for $\mathcal{U},\mathcal{P}_{+},\mathcal{P}_{-},\mathcal{K}_{2,0},\mathcal{K}_{1,1},\mathcal{K}_{0,2}$, and $\mathcal{T}$, where the pink curves represent two common curves.}
\label{fig:efficient defining pairs of standard surfaces}
\end{figure}

\label{ex:Kirby-Thompson of standard surfaces}
\end{ex}

%主定理の証明

\section{Main theorem}

In this section, we prove Theorem \ref{thm:main theorem}, where we determine the Kirby-Thompson invariants of distant sums of finite standard unknotted surfaces including the standard tori. We prepare some lemmas and prove the main theorem in Subsection 4.1. Finally, in Subsection 4.2, we see some figures used to improve upper bounds of these invariants.

\subsection{Proof of the main theorem}

\begin{thm}
If $F$ is a finite distant sum of standard unknotted surfaces in $S^{4}$ and $n(F)$ is the number of standard unknotted tori included in $F$, then $\mathcal{L}(F)=\mathcal{L}^{*}(F)=3n(F).$
\label{thm:main theorem}
\end{thm}

\begin{proof}
From Lemma \ref{lem:upper bound}, it is sufficient to prove $\mathcal{L}^{*}(\mathcal{T})\geq3n(F)$ for any unstabilized bridge trisection $\mathcal{T}$ of $F$. If $n(F)=0$, then the above equalities hold from Figure \ref{fig:efficient1}, Figure \ref{fig:efficient2}, and Figure \ref{fig:efficient3} in Subsection 4.2. When $n(F)>0$, Lemma \ref{lem:main lemma 2} implies that $\sum_{i\not=j}d^{*}(p_{ij}^{i},p_{ki}^{i})\geq3n(F)$ holds for any unstabilized bridge trisection $\mathcal{T}$ of $F$ and efficient defining pair $(p_{ij}^{i},p_{ki}^{i})$ for the bridge trisection, completing the proof of the theorem.
\end{proof}

From Example \ref{ex:Kirby-Thompson of standard surfaces} and Theorem \ref{thm:main theorem}, the following corollary holds.

\begin{cor}
The Kirby-Thompson invariants of knotted surfaces are additive with respect to the distant sum operations between standard surfaces. In other words, the following equations hold, where each $F_{i}$ is standard:
\[
\mathcal{L}\left(\coprod_{i=1}^{n}F_{i}\right)=\sum_{i=1}^{n}\mathcal{L}(F_{i}),\,\mathcal{L}^{*}\left(\coprod_{i=1}^{n}F_{i}\right)=\sum_{i=1}^{n}\mathcal{L}^{*}(F_{i}).
\]
\end{cor}

\begin{prob}
Are the Kirby-Thompson invariants of knotted surfaces additive with respect to the distant sum operations between any knotted surfaces?
\end{prob}

In order to prove Theorem \ref{thm:main theorem}, we prepare some lemmas. Let $L=\alpha_{+}\cup_{\Sigma}\alpha_{-}$ be a $b$-bridge split unlink with $b\geq2$ and $\Sigma$ denotes the bridge surface.

\begin{lem}
Let $(p_{+},p_{-})\in{P_{c}(\alpha_{+})\times P_{c}(\alpha_{-})}$ be an efficient defining pair for the bridge split unlink $L$. We write $p_{+}=\psi\cup g_{+}$ and $p_{-}=\psi\cup g_{-}$ as in Lemma \ref{lem:efficient defining pair 1} so that they satisfy all properties in the lemma. Then, any essential curve in $g_{\pm}$ bounds an even number of punctures of $L\cap\Sigma$.
\label{lem:efficient defining pair 2}
\end{lem}

\begin{proof}

Suppose that an essential curve $\gamma\in{p_{+}}$ bounds an odd number of punctures of $\Sigma$. From Lemma 5.10 of \cite{Kirby_Thompson_1}, $\gamma$ is a cut-reducing curve for the unlink $L$, following that $\gamma\in{\psi=p_{+}\cap p_{-}}$ from Lemma \ref{lem:efficient defining pair 1}.

\end{proof}

\begin{lem}
Let $\theta\in{p_{ij}^{i}}$ be a compressing curve for the unlink $L=\alpha_{+}\cup_{\Sigma}\alpha_{-}$ and $\pi_{\theta}^{\text{in}}\subset{\Sigma}$ be the pair of pants with $\partial\pi_{\theta}^{\text{in}}=\theta\cup\partial_{1}\cup\partial_{2}$ and $\pi_{\theta}^{\text{in}}\subset{S_{\theta}^{\text{in}}}$.
\begin{itemize}
\item[(i)] If $\theta$ is reducing for the unlink $L$, then both $\partial_{1}$ and $\partial_{2}$ are either reducing or cut-reducing for $L$.
\item[(ii)] If $\theta$ is compressing for $\alpha_{+}$ but for $\alpha_{-}$, then both $\partial_{1}$ and $\partial_{2}$ are cut-reducing for $L$.
\end{itemize}
\label{lem:efficient defining pair 3}
\end{lem}

\begin{proof}
First, we show (i). Suppose that $\partial_{1}$ is compressing for $\alpha_{+}$ but for $\alpha_{-}$ (i.e., $\partial_{1}\in{g_{+}}$). Since $\theta$ bounds an even number of punctures, $\partial_{2}\in{g_{+}}$ also bounds an even number of punctures. From Lemma \ref{lem:efficient defining pair 1}, $\theta$ is not moved and both $\partial_{1}$ and $\partial_{2}$ are moved by length one edges to $\partial_{1}'$ and $\partial_{2}'$ in $g_{-}$, respectively. Suppose that the A-move $\partial_{2}\mapsto\partial_{2}'$ occurs after the A-move $\partial_{1}\mapsto\partial_{1}'$ occurs. $E\subset{\Sigma}$ denotes the 4 punctured 2-sphere where the A-move $\partial_{1}\mapsto\partial_{1}'$ occurs, following that $\partial_{2}\cup\theta\subset{\partial E}$. Let $E'\subset{\Sigma}$ denote the 4 punctured 2-sphere where $\partial_{2}\mapsto\partial_{2}'$ occurs. From Lemma \ref{lem:efficient defining pair 1}, there exists some boundary component $\partial$ satisfying $\partial=\partial\pi_{\partial_{1}}^{\text{in}}\cap\partial E'$. Since $\partial_{2}$ does not bound $\partial$, $\partial_{2}'$ bounds $\partial$ after the second A-move $\partial_{2}\mapsto\partial_{2}'$. Consequently, $\partial$ is bounded by both $\partial_{1}'$ and $\partial_{2}'$, and it follows that $|\partial_{1}\cap\partial_{1}'|=|\partial_{1}\cap\partial_{2}'|=2$. This contradicts to (iv) of Lemma \ref{lem:efficient defining pair 1}. 

%FIGURE9

\begin{figure}[h]
\center
\includegraphics[width=50mm]{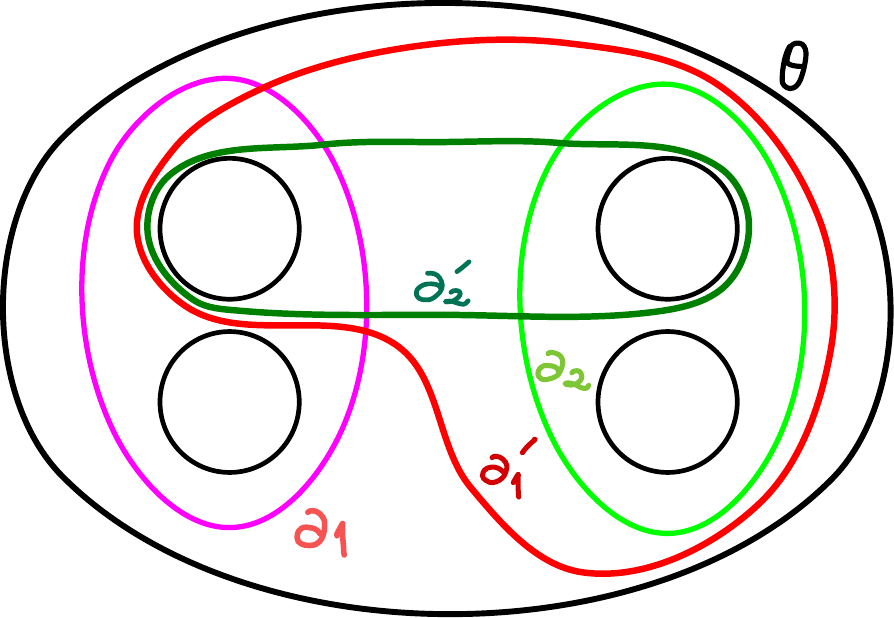}
\caption{An image of A-moves $\partial_{1}\mapsto\partial_{1}'$ and $\partial_{2}\mapsto\partial_{2}'$.}
\end{figure}

Finally, we show (ii). Suppose that $\theta\in{g_{+}}$ and $\partial_{1}$ is also compressing for $\alpha_{+}$, where $\partial\pi_{\theta}^{\text{in}}=\theta\cup\partial_{1}\cup\partial_{2}$. Since each $\theta$ and $\partial_{1}$ bounds an even number of punctures, $\partial_{2}$ also bounds an even number of punctures.

Suppose that $\partial_{2}$ is reducing for $L$, following that it is not moved by $p_{+}\rightarrow p_{-}$ from Lemma \ref{lem:efficient defining pair 1}. Hence, it is sufficient to prove the following two cases.\\
\textbf{Case1} $\partial_{1}\mapsto\partial_{1}'$ occurs before $\theta\mapsto\theta'$ occurs. \\
\textbf{Case2} $\theta\mapsto\theta'$ occurs before $\partial_{1}\mapsto\partial_{1}'$ occurs.

However, we obtain the same contradictions as in the proof of (i) (see Figure 10).

%FIGURE10

\begin{figure}[h]
\center
\includegraphics[width=90mm]{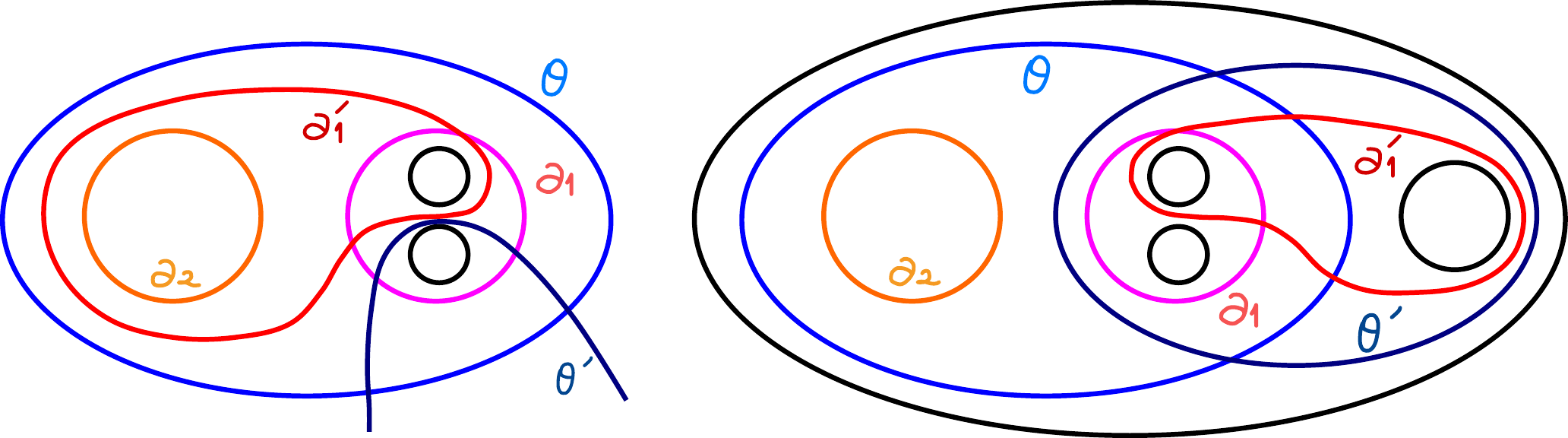}
\caption{An image of A-moves $\theta\mapsto\theta'$ and $\partial_{1}\mapsto\partial_{1}'$.}
\end{figure}

Finally, suppose that $\partial_{2}$ is compressing for $\alpha_{+}$. Let $\theta\mapsto\theta'$, $\partial_{1}\mapsto\partial_{1}'$, and $\partial_{2}\mapsto\partial_{2}'$ denote the three A-moves. If $\theta\mapsto\theta'$ occurs after $\partial_{1}\mapsto\partial_{1}'$ and $\partial_{2}\mapsto\partial_{2}'$, this is contradiction as in the proof of (i). If $\theta\mapsto\theta'$ occurs after $\partial_{1}\mapsto\partial_{1}'$, then we obtain the same contradiction as in the proof of when $\partial_{2}$ is reducing.

\end{proof}

Let $F=F_{1}\sqcup\cdots\sqcup F_{n}$ be the distant sum of standard unknotted surfaces with $n\geq2$ and $\mathcal{T}$ an unstabilized bridge trisection of $F$ whose spine and bridge surface are $\bigcup_{i\not=j}(B_{ij},\alpha_{ij})$ and $\Sigma$, respectively. Let $(p_{ij}^{i},p_{ki}^{i})\in{P_{c}(\alpha_{ij})\times P_{c}(\alpha_{ki})}$ be an arbitrary efficient defining pair for the unlink $\partial\mathcal{D}_{i}=\alpha_{ij}\cup_{\Sigma}\alpha_{ki}$.

\begin{lem}
Let $c\in{\psi_{i}}=p_{ij}^{i}\cap p_{ki}^{i}$ be a cut-reducing curve. Let $\Pi_{l}:=F_{l}\cap\Sigma=:\{a_{1}^{l},\cdots,a_{m_{l}}^{l}\}$, where $m_{l}=2b(F_{l})$ for each $l\in{\{1,\cdots,n\}}$. 
\begin{itemize}
\item[(i)] $c$ bounds an odd number of punctures
\[
\bigcup_{l=1}^{L}\Pi_{l}\cup\left\{a_{1}^{L+1}\cdots,a_{2M-1}^{L+1}\right\},
\]
where $0\leq L<n$, $M\in{\{1,2,3\}}$, and $\bigcup\Pi_{l}=\phi$ when $L=0$.
\item[(ii)] Let $\pi_{c}^{\text{in}}\subset{\Sigma}$ be the pair of pants with $\partial\pi_{c}^{\text{in}}=c\cup\partial_{1}\cup\partial_{2}$. Then, there are six possible isotopy classes of the pair $(\partial_{1},\partial_{2})$ as in Figure \ref{fig:lemma 4 1 3}.

%FIGURE11

\begin{figure}[h]
\center
\includegraphics[width=80mm]{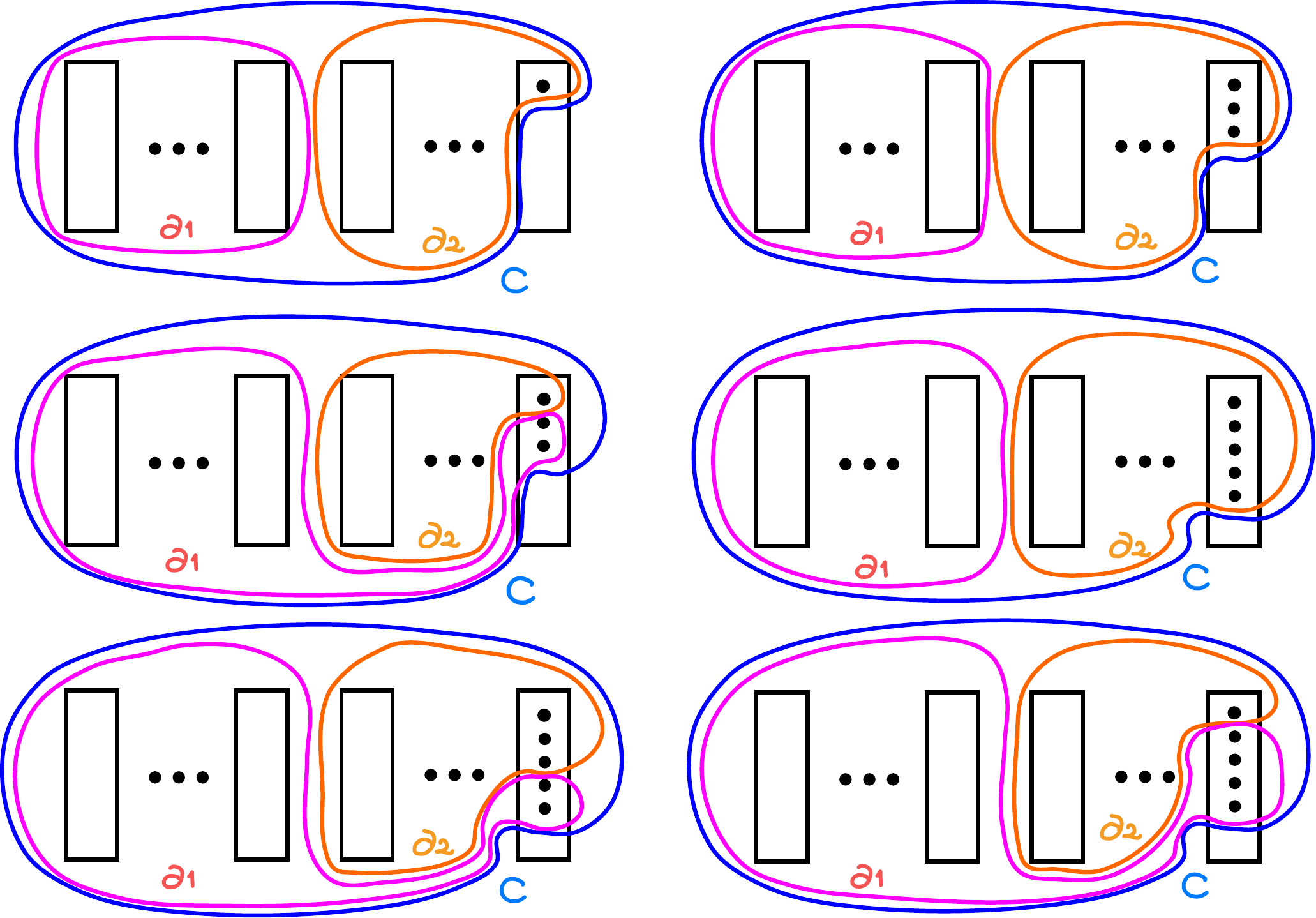}
\caption{Each box in the above figure represents the set of punctures $\Pi_{l}=F_{l}\cap\Sigma$. The blue and the orange curves represent two cut-reducing curves and the pink curves represent either a compressing curve or a reducing curve for $\partial\mathcal{D}_{i}$.}
\label{fig:lemma 4 1 3}
\end{figure}

\end{itemize}
\label{lem:efficient defining pair 4}
\end{lem}

\begin{proof}
First, we show (i). Let $c\in{\psi_{i}=p_{ij}^{i}\cap p_{ki}^{i}}$ be a cut-reducing curve for $\partial\mathcal{D}_{i}$, following that $c$ bounds an odd number of punctures of $\Pi:=\bigcup_{l}\Pi_{l}\subset{\Sigma}$. Let $\partial\pi_{c}^{\text{in}}=:c\cup\partial_{1}\cup\partial_{2}$. Then, we can assume that $\partial_{1}$ bounds an odd number of punctures and $\partial_{2}$ bounds an even number of punctures. Suppose that $c$ bounds an odd number of punctures of $\Pi_{L+1}$. If $c$ also bounds an odd number of punctures of $\Pi_{l'}$ for some integer $l'\not=L+1$, this contradicts to that $c$ is cut-reducing for $\partial\mathcal{D}_{i}$. If $c$ bounds an even number of punctures $\Pi_{l'}'\subsetneq\Pi_{l'}$ for $l'\not=L+1$, it follows $|F_{l'}\cap\partial\mathcal{D}_{i}|\geq2$, and this is a contradiction. Hence, it follows that, if $c$ bounds other punctures of $\Pi_{l}$, then $c$ must bound all punctures of $\Pi_{l}$ and this completes the proof of (i).

Finally, we show (ii). Assume that a cut-reducing curve $c$ bounds an odd number of punctures $\bigcup_{l=1}^{L}\Pi_{l}\cup\{a_{1}^{L+1},\cdots,a_{2M-1}^{L+1}\}$ using (i). If $L=0$, $c$ bounds $\{a_{1}^{1},\cdots,a_{2M-1}^{1}\}$, where $M\in{\{1,2,3\}}$, and one of $\{\partial_{1},\partial_{2}\}$ is compressing and the other is cut, completing the proof. Assume that $L\geq1$, $\partial\pi_{c}^{\text{in}}=:c\cup\partial_{1}\cup\partial_{2}$, and $\partial_{1}$ bounds an odd number of punctures. If $\partial_{1}$ bounds an even number of punctures of $\{a_{1}^{L+1},\cdots,a_{2M-1}^{L+1}\}$ including zero, $\partial_{2}$ bounds an odd number of punctures of $\{a_{1}^{L+1},\cdots,a_{2M-1}^{L+1}\}$, contradicting to that $\partial_{2}$ is compressing for $\partial\mathcal{D}_{i}$. Hence, it follows that $\partial_{1}$ bounds $\{a_{1}^{L+1},\cdots,a_{2M-1}^{L+1}\}$. For some $l\not=L+1$, if $\partial_{1}$ decomposes $\Pi_{l}$ into two sets of even punctures $\Pi_{l}'$ and $\Pi_{l}''$, then it follows $|F_{l}\cap\partial\mathcal{D}_{i}|\geq2$ and this contradicts to Lemma \ref{lem:bridge trisection of distant sum}. Consequently, we obtain the six cases in Figure \ref{fig:lemma 4 1 3}.
\end{proof}

\begin{lem}
Let $c\in{g_{ij}^{i}}$ be a compressing curve for $\alpha_{ij}$. Let $\Pi_{l}:=F_{l}\cap\Sigma=:\left\{a_{1}^{l},\cdots,a_{m_{l}}^{l}\right\}$, where $m_{l}=2b(F_{l})$ for each $l\in{\{1,\cdots,n\}}$.

\begin{itemize}
\item[(i)] $c$ bounds an even number of punctures
\[
\bigcup_{l=1}^{L}\Pi_{l}\cup\left\{a_{1}^{L+1},\cdots,a_{2M}^{L+1}\right\},
\]
where $0\leq L<n$, $M\in{\{1,2\}}$, and $\bigcup\Pi_{l}=\phi$ when $L=0$.
\item[(ii)] Let $\pi_{c}^{\text{in}}\subset{\Sigma}$ be the pair of pants with $\partial\pi_{c}^{\text{in}}=c\cup\partial_{1}\cup\partial_{2}$. Then, there are two possible isotopy classes of the pair $(\partial_{1},\partial_{2})$ as in the following figure.

%FIGURE12

\begin{figure}[h]
\center
\includegraphics[width=80mm]{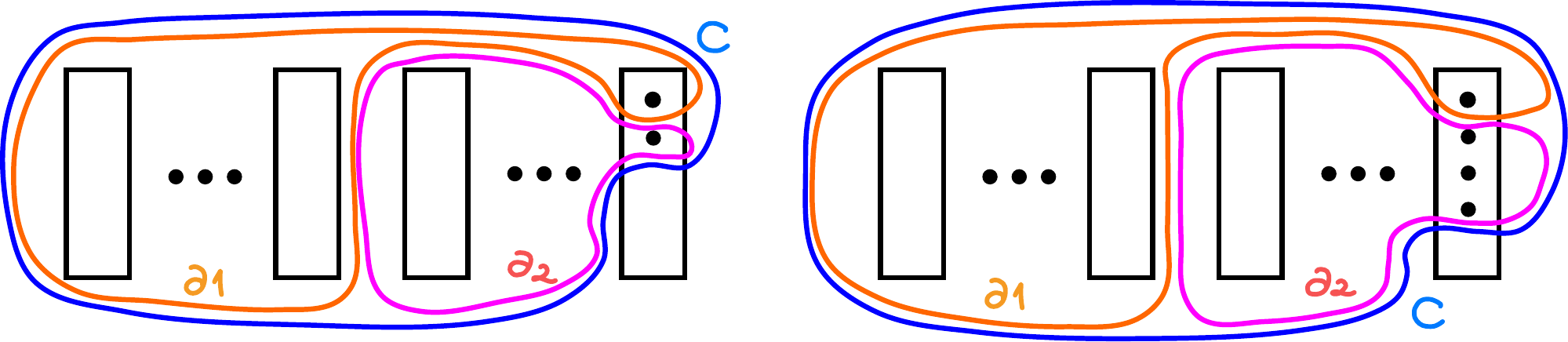}
\caption{The blue curve represents a compressing curve and the others represent cut-reducing curves.}
\label{fig:lemma 4 1 4}
\end{figure}

\end{itemize}

\label{lem:efficient defining pair 5}
\end{lem}

\begin{proof}
First, we show (i). Let $c\in{g_{ij}^{i}}$ be a compressing curve for $\alpha_{ij}^{i}$, then it follows easily that it bounds an even number of punctures of $\Pi:=\bigcup_{l}\Pi_{l}$. Since $c\in{g_{ij}^{i}}$, if $c$ bounds some punctures $\Pi_{l}'\subset{\Pi_{l}}$, then $|\Pi_{l}'|$ must be even for any integer $l$. Suppose that $c$ bounds $\Pi_{l}'\subsetneq\Pi_{l}$ and $\Pi_{m}'\subsetneq\Pi_{m}$ for some integers $l\not=m$ and $\partial\pi_{c}^{\text{in}}=:c\cup\partial_{1}\cup\partial_{2}$.
From Lemma \ref{lem:efficient defining pair 1}, it follows that both $\partial_{1}$ and $\partial_{2}$ bound an odd number of punctures and they are not moved by $p_{ij}^{i}\rightarrow p_{ki}^{i}$. If $\partial_{1}$ bounds all punctures of $\Pi_{l}'$, this contradicts to that the bridge trisection of $F$ is unstabilized, and $\partial_{1}$ cannot bound an even number of punctures $\Pi_{l}''\subsetneq\Pi_{l}'$ from the same reason. Hence, $\partial_{1}$ and $\partial_{2}$ bound an odd number of punctures $\Pi_{l}''\subsetneq\Pi_{l}'$ and $\Pi_{l}'''\subsetneq\Pi_{l}'$, respectively, where $\Pi_{l}''\cup\Pi_{l}'''=\Pi_{l}'$. On the other hand, since $c$ also bounds $\Pi_{m}'\subsetneq{\Pi_{m}}$, $\partial_{1}$ and $\partial_{2}$ bound an odd number of punctures $\Pi_{m}''\subsetneq\Pi_{m}'$ and $\Pi_{m}'''\subsetneq\Pi_{m}'$, respectively, where $\Pi_{m}''\cup\Pi_{m}'''=\Pi_{m}'$. In this situation, we obtain a contradiction that $\partial_{*}$ are efficient. 

Finally, we show (ii). From (i), we suppose that $c$ bounds an even number of punctures $\bigcup_{l=1}^{L}\Pi_{l}\cup\{a_{1}^{L+1},\cdots,a_{2M}^{L+1}\}$, where $0\leq L<n$ and $M\in{\{1,2\}}$. Since $c$ bounds an even number of punctures, both $\partial_{1}$ and $\partial_{2}$ bound either an even punctures or an odd punctures. From Lemma \ref{lem:efficient defining pair 3}, they must bound an odd number of punctures and we obtain two cases in Figure \ref{fig:lemma 4 1 4}.

\end{proof}

\begin{lem}
Let $c\in{\psi_{i}}$ be a reducing curve for $\partial\mathcal{D}_{i}=\alpha_{ij}\cup_{\Sigma}\alpha_{ki}$. Let $\partial\pi_{c}^{\text{in}}=c\cup\partial_{1}\cup\partial_{2}$. Then, both $\partial_{1}$ and $\partial_{2}$ are either reducing or cut-reducing for $\partial\mathcal{D}_{i}$ as in the figure below.

%FIGURE13

\begin{figure}[h]
\center
\includegraphics[width=90mm]{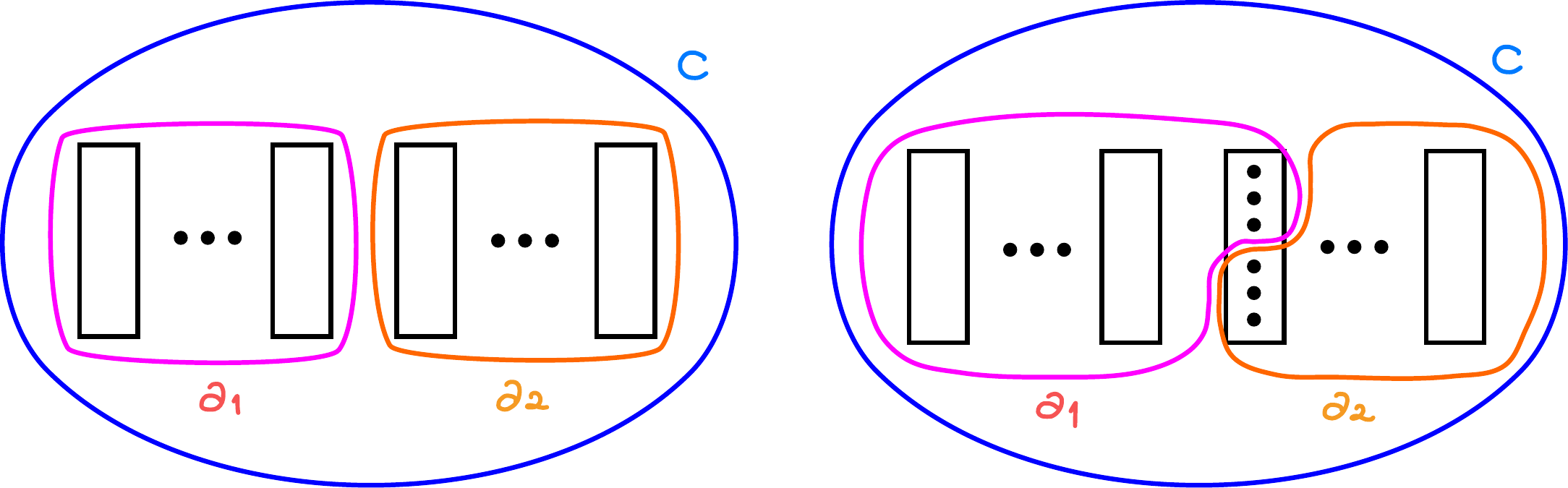}
\caption{An image of three curves $c$, $\partial_{1}$, and $\partial_{2}$.}
\label{fig:lemma 4 1 5}
\end{figure}

\label{lem:efficient defining pair 6}
\end{lem}

\begin{proof}
Let $c\in{\psi_{i}}$ be a reducing curve for $\partial\mathcal{D}_{i}$ and $\partial\pi_{c}^{\text{in}}=c\cup\partial_{1}\cup\partial_{2}$. Since $c$ bounds an even number of punctures of $\Sigma$, both $\partial_{1}$ and $\partial_{2}$ bound either even punctures or odd punctures. If $\partial_{1}$ and $\partial_{2}$ bound an odd number of punctures, from Lemma \ref{lem:efficient defining pair 4}, both of them are cut-reducing for the unlink $\partial\mathcal{D}_{i}$ as in the right side of the above figure. If $\partial_{1}$ and $\partial_{2}$ bound an even number of punctures, then both of them must be reducing for $\partial\mathcal{D}_{i}$ from Lemma \ref{lem:efficient defining pair 3}.
\end{proof}

Let $(p_{ij}^{i},p_{ki}^{i})\in{P_{c}(\alpha_{ij})\times P_{c}(\alpha_{ki})}$ be an arbitrary efficient defining pair for the unlink $\partial\mathcal{D}_{i}=\alpha_{ij}\cup_{\Sigma}\alpha_{ki}$, where each $\alpha_{ij}$ forms the spine of an unstabilized bridge trisection $\mathcal{T}$ of the distant sum of finite standard surfaces $F=F_{1}\sqcup\cdots\sqcup F_{n}$. Let decompose $\Sigma$ into the upper and lower disks $D_{+}^{2}$ and $D_{-}^{2}$ and suppose that all punctures $F\cap\Sigma$ and every curve $c_{l}\in{c}\in{p_{ij}^{i}\cup p_{ki}^{i}}$ are inside $D_{+}^{2}$. Then, there exists the 4 punctured 2-sphere $E\subset{D_{+}^{2}}$ such that $\partial E=\partial D_{+}^{2}\cup\partial_{1}\cup\partial_{2}\cup\partial_{3}$ and $E\cap(F\cap\Sigma)=\phi$, where each $\partial_{i}$ is either some essential simple closed curve in $p_{ij}^{i}\cup p_{ki}^{i}$ or the boundary of a small regular neighborhood of some puncture of $F\cap\Sigma$. Then, possible isotopy classes of the tuple $\{\partial_{1},\partial_{2},\partial_{3}\}$ are the followings.\\
\textbf{Case1} All of them are reducing for $\partial\mathcal{D}_{i}$.\\
\textbf{Case2} One of them is reducing and the others are cut-reducing for $\partial\mathcal{D}_{i}$.\\
\textbf{Case3} $\partial_{1}$ is compressing for $\alpha_{ij}$ and $\partial_{2}$ and $\partial_{3}$ are cut-reducing for $\partial\mathcal{D}_{i}$.

By the way, we note that it is impossible that at most two curves of $\{\partial_{1},\partial_{2},\partial_{3}\}$ are compressing for $\alpha_{ij}$ from Lemma \ref{lem:efficient defining pair 3}. For simplicity, we can assume that possible cases of $\{\partial_{1},\partial_{2},\partial_{3}\}$ are either \textbf{Case1} or \textbf{Case2} from the following lemma.

\begin{lem}
Let $\{\partial_{1},\partial_{2},\partial_{3}\}$ be essential simple closed curves in $p_{ij}^{i}\cup p_{ki}^{i}$ mentioned above. Then, through an appropreate homeomorphism on $\Sigma-F$, we can assume that possible isotopy classes of the tuple $\{\partial_{1},\partial_{2},\partial_{3}\}$ are either \textbf{Case1} or \textbf{Case2}.
\label{lem:efficient defining pair }
\label{lem:efficient defining pair 7}
\end{lem}

\begin{proof}
Suppose that $\partial_{1}$ is compressing for $\alpha_{ij}$ and $\partial_{2}$ and $\partial_{3}$ are cut-reducing for $\partial\mathcal{D}_{i}$. For simplicity, let $\partial_{1}$ bounds an even number of punctures $\Pi_{1}'\subsetneq\Pi_{1}:=F_{1}\cap\Sigma$, following that $\partial_{2}$ and $\partial_{3}$ bound the remaining punctures $\Pi_{1}-\Pi_{1}'$. Let $\partial\pi_{\partial_{2}}^{\text{in}}=\partial_{2}\cup\partial_{1}'\cup\partial_{2}'$ and assume that $\partial_{2}'$ is cut-reducing for $\partial\mathcal{D}_{i}$, following that $\partial_{2}'$ is cut-reducing and $\partial_{1}^{'}$ is either reducing or compressing from Lemma \ref{lem:efficient defining pair 4} as in Figure 14. If $\partial_{1}'$ is reducing for $\partial\mathcal{D}_{i}$, then we can transform $\partial_{2}$ to $\widehat{\partial_{2}}$ in $\Sigma-F$, obtaining the \textbf{Case2} where $\partial_{1}'$ is reducing and both $\partial_{2}'$ and $\widehat{\partial_{2}}$ are cut-reducing. Suppose that $\partial_{1}'$ is compressing for $\alpha_{ij}$. First, we transform the cut-reducing curve $\partial_{2}$ to $\widehat{\partial_{2}}$ so that $\partial E'=\partial D_{+}^{2}\cup\widehat{\partial_{2}}\cup\partial_{1}'\cup\partial_{2}'$, where $E'$ is the new punctured sphere after the transformation. After that, we transform the cut-reducing curve $\partial_{2}'$ to $\widehat{\partial_{2}'}$ in $\Sigma-F$ so that $\partial E''=\partial D_{+}^{2}\cup\widehat{\partial_{2}'}\cup x\cup y$, where $\partial\pi_{\partial_{2}'}^{\text{in}}=\partial_{2}'\cup x\cup y$. Consequently, we obtain the \textbf{Case2} where $x$ is reducing and $y$ and $\partial_{2}''$ are cut-reducing.

%FIGURE14

\begin{figure}[t]
\center
\includegraphics[width=120mm]{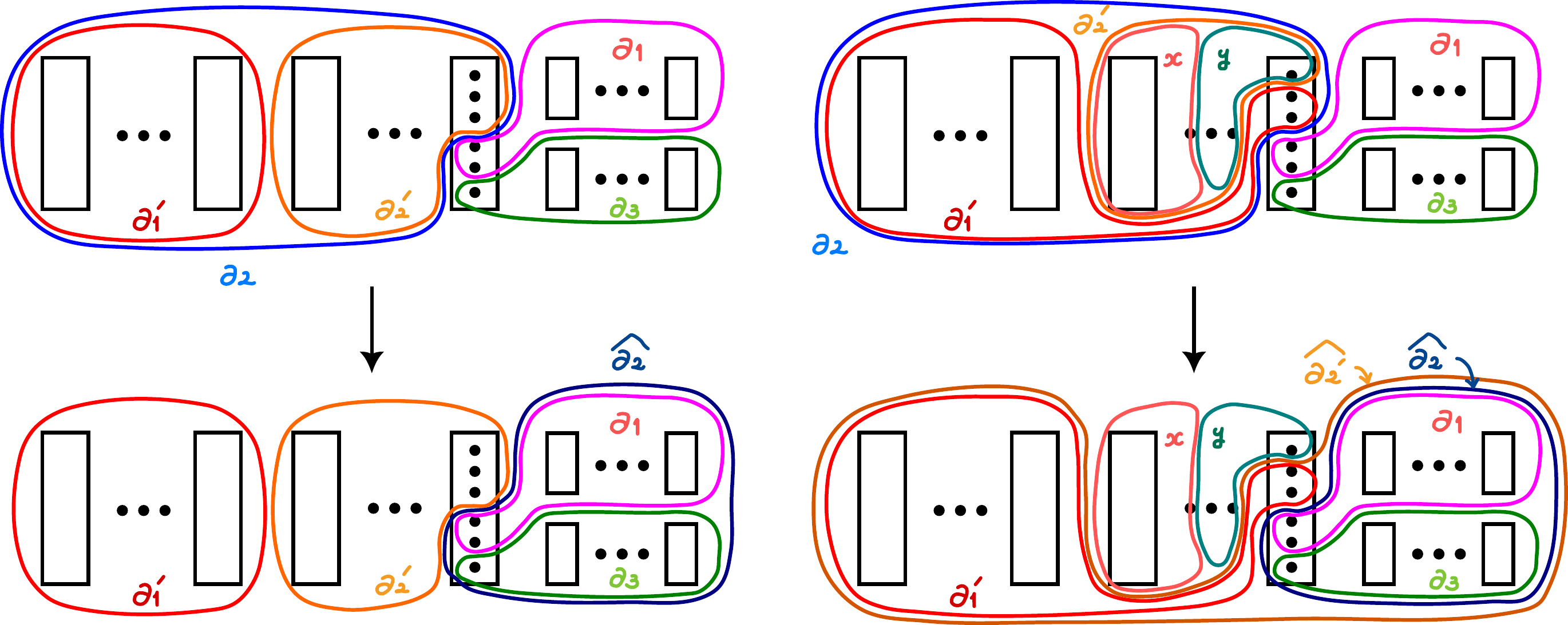}
\caption{Appropreate homeomorphisms on $\Sigma-F$ transforning the three curves in \textbf{Case3} to ones in \textbf{Case1} or \textbf{Case2}.}
\label{fig:lemma 4 1 6}
\end{figure}

\end{proof}

\begin{lem}
Let $F$ be a finite distant sum of standard unknotted surfaces. Let $n(F)>0$ denote the number of standard tori in $F$ and $T_{1},\cdots,T_{n(F)}$ the standard tori in $F$. Let $(p_{ij}^{i},p_{ki}^{i})\in{P_{c}(\alpha_{ij})\times P_{c}(\alpha_{ki})}$ be an efficient defining pair for the unlink $\partial\mathcal{D}_{i}=\alpha_{ij}\cup_{\Sigma}\alpha_{ki}$, where each $\alpha_{ij}$ is the trivial tangle forming the spine of an unstabilized bridge trisecion of $F$. For each $i\in{\{1,2,3\}}$ and $l\in{\{1,\cdots,n(F)\}}$, there exist a cut-reducing curve $\gamma_{l}^{i}\in{\psi_{i}}$ and a cut-reducing 2-sphere $Q_{\gamma_{l}^{i}}\subset{\partial X_{i}}$ for the unlink $\partial\mathcal{D}_{i}$ satisfying
\begin{itemize}
\item[(i)] $Q_{\gamma_{l}^{i}}\cap\Sigma=\gamma_{l}^{i}$,
\item[(ii)] $|Q_{\gamma_{l}^{i}}\cap\partial\mathcal{D}_{i}|=|Q_{\gamma_{l}^{i}}\cap(\partial\mathcal{D}_{i}\cap T_{l})|=2$, and
\item[(iii)] each $\gamma_{l}^{i}$ bounds exactly three punctures of $T_{l}\cap\Sigma$.
\end{itemize}
\label{lem:main lemma 1}
\end{lem}

\begin{proof}
Let $T\in\{T_{1},\cdots,T_{n(F)}\}$ be an arbitrary torus component in $F$ and $(p_{ij}^{i},p_{ki}^{i})\in{P_{c}(\alpha_{ij})\times P_{c}(\alpha_{ki})}$ be an arbitrary efficient defining pair for $\partial\mathcal{D}_{i}$. Let $\Pi_{l}:=T_{l}\cap\Sigma$ for each $l\in{\{1,\cdots,n(F)\}}$ and $\Pi_{T}:=T\cap\Sigma$. Let $c\in{p_{ij}^{i}\cup p_{ki}^{i}}$ be an arbitrary essential simple closed curve, following, from Lemma \ref{lem:efficient defining pair 4}, Lemma \ref{lem:efficient defining pair 5}, and Lemma \ref{lem:efficient defining pair 6}, that it bounds the following punctures of $\Pi:=F\cap\Sigma$:
\[
\bigcup_{l=1}^{L}\Pi_{l}\cup\left\{a_{1}^{L+1},\cdots,a_{m_{L+1}}^{L+1}\right\}\subset{\Pi},
\]
where $0\leq L<n-1$, $1\leq m_{l}\leq|\Pi_{l}|$, and $\bigcup\Pi_{l}=\phi$ if $L=0$. We show this lemma in two cases.\\
\\
\textbf{Case1} $\{a_{1}^{L+1},\cdots,a_{m_{L+1}}^{L+1}\}\subsetneq\Pi_{T}$.

Let $\{\partial_{1},\partial_{2},\partial_{3}\}$ be c-reducing curves for the unlink $\partial\mathcal{D}_{i}$ as in Lemma \ref{lem:efficient defining pair 7}, i.e., $\partial_{i}\in{p_{ij}^{i}\cup p_{ki}^{i}}$ and $\Pi\subset S_{\partial_{1}}^{\text{in}}\cup S_{\partial_{2}}^{\text{in}}\cup S_{\partial_{3}}^{\text{in}}$ hold.\\
\textbf{Subcase1a} $c\in{\{\partial_{1},\partial_{2},\partial_{3}\}}$.

%FIGURE15

\begin{figure}[h]
\center
\includegraphics[width=60mm]{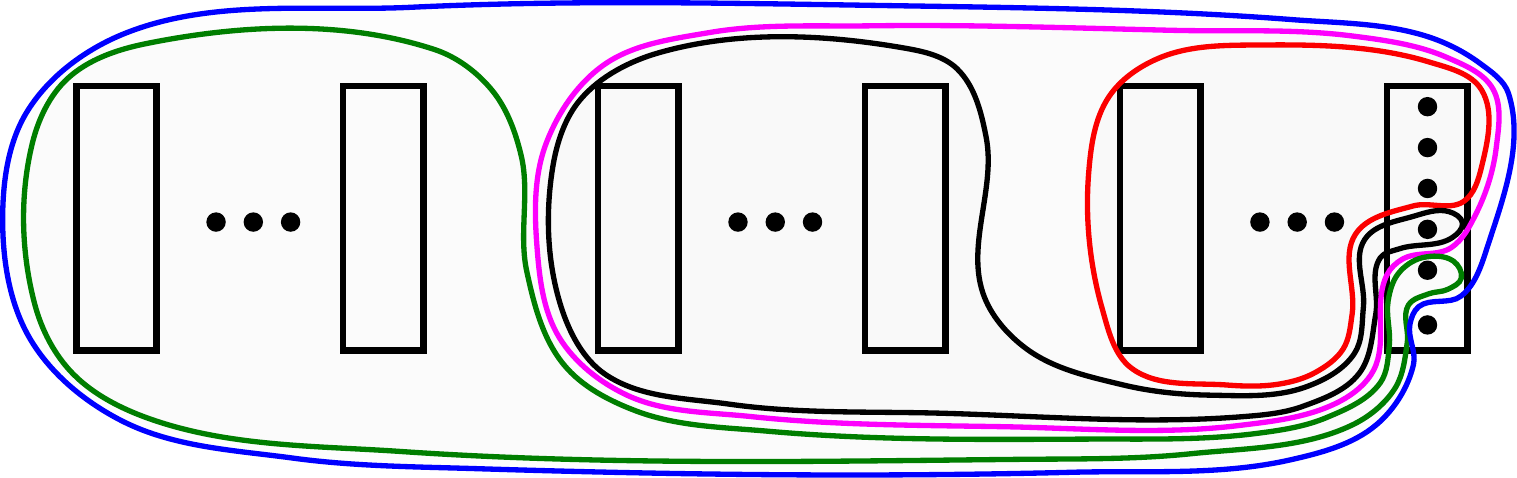}
\caption{The rightmost box represents the set of punctures $\Pi_{T}$ and the red curve is the desired cut-reducing curve bounding exactly three punctures of $\Pi_{T}$.}
\label{fig:lemma 4 1 7 1}
\end{figure}

From Lemma \ref{lem:efficient defining pair 7}, we know that $m_{L+1}\in{\{1,3,5\}}$ and $c$ is cut-reducing for $\partial\mathcal{D}_{i}$. Without loss of generality, we can assume $c=\partial_{1}$. If $m_{L+1}=1$, then either $\partial_{2}$ or $\partial_{3}$ must bound the remaining five punctures of $\Pi_{T}$ from Lemma \ref{lem:efficient defining pair 6} and it is sufficient to prove the cases $m_{L+1}\in{\{3,5\}}$. If $m_{L+1}=3$, then $c=\partial_{1}$ bounds exactly three punctures of $\Pi_{T}$ and we have the desired result. Now, let $m_{L+1}=5$ and show the existance of some cut-reducing curve inside $c$ bounding exactly three punctures of $\Pi_{T}$ by induction with respect to $0\leq L<n-1$.

 If $L=0$, then we know that, inside $c$, there exists some cut-reducing curve bounding three punctures of $\Pi_{T}$ from Lemma \ref{lem:efficient defining pair 3}, Lemma \ref{lem:efficient defining pair 4}, and Lemma \ref{lem:efficient defining pair 5}. Suppose that the claim holds for some integer $0\leq L<n-2$ and $c$ bounds $\bigcup_{l=1}^{L+1}\Pi_{l}\cup\left\{a_{1}^{L+2},\cdots,a_{5}^{L+2}\right\}\subset{\Pi}$. From Lemma \ref{lem:efficient defining pair 4}, letting $\partial\pi_{c}^{\text{in}}=:c\cup\partial_{1}'\cup\partial_{2}'$ and $\partial_{1}'$ be cut-reducing, $\partial_{1}'$ bounds either one, three, or five punctures of $\{a_{1}^{L+2},\cdots,a_{5}^{L+2}\}$. If it bounds three, then $\partial_{1}'$ is the desired curve, and if it bounds five, then we know that there exists some cut-reducing curve inside $\partial_{1}'$ bounding exactly three punctures of $\{a_{1}^{L+2},\cdots,a_{5}^{L+2}\}$ using the induction hypothesis. If $\partial_{1}'$ bounds one puncture of $\{a_{1}^{L+2},\cdots,a_{5}^{L+2}\}$, then it follows that $\partial_{2}'$ is a compressing curve for $\alpha_{ij}$ bounding the remaining four punctures of $\{a_{1}^{L+2},\cdots,a_{5}^{L+2}\}$. From Lemma \ref{lem:efficient defining pair 5}, letting $\partial\pi_{\partial_{2}'}^{\text{in}}=:\partial_{2}'\cup\partial_{1}''\cup\partial_{2}''$, either $\partial_{1}''$ or $\partial_{2}''$ must bound exactly three punctures of $\Pi_{T}$, obtaining the desired cut-reducing curve (see the Figure \ref{fig:lemma 4 1 7 1}).\\
\textbf{Subcase1b} $c\notin{\{\partial_{1},\partial_{2},\partial_{3}\}}$.

Without loss of generality, we can assume that $c$ is inside $\partial_{1}$, where $\partial_{1}$ is one of the outer most curves included in $D_{+}^{2}$. Then, there exist essential simple closed curves $\partial_{1}', \partial_{2}'\in{p_{ij}^{i}\cup p_{ki}^{i}}$ such that $c\cup\partial_{2}'\subset S_{\partial_{1}'}^{\text{in}}\subseteq S_{\partial_{1}}^{\text{in}}$. Let $0\leq n_{L+1}<6$ denote the number of punctures of $\Pi_{T}$ bounded by the $\partial_{2}'$, following that $\partial_{1}'$ bounds $0<m_{L+1}+n_{L+1}\leq6$ punctures of $\Pi_{T}$, where $0<m_{L+1}<6$ denotes the number of punctures of $\Pi_{T}$ bounded by $c$. We note that, from Lemma \ref{lem:efficient defining pair 3}, either $m_{L+1}$ or $n_{L+1}$ is odd. If $m_{L+1}=3$, $n_{L+1}=3$, or $m_{L+1}+n_{L+1}=3$ holds, then $c$, $\partial_{2}'$, or $\partial_{1}'$ is a desired cut-reducing curve. If $m_{L+1}=5$ or $n_{L+1}=5$ holds, then we know that there exists a desired curve bounding exactly three punctures of $\Pi_{T}$ inside $c$ or $\partial_{2}'$ as in \textbf{Subcase1a}. If $m_{L+1}=4$ or $n_{L+1}=4$ holds, then we can find a desired cut-reducing curve inside $c$ or $\partial_{2}'$ as in \textbf{Subcase1a}. 

The remaining cases are $(m_{L+1},n_{L+1})\in{\{(1,0),(1,1)\}}$. We show the case $(m_{L+1},n_{L+1})=(1,0)$. Let $\gamma_{1},\gamma_{2},\cdots,\gamma_{k}\in{p_{ij}^{i}\cup p_{ki}^{i}}$ be essential curves satisfying $\gamma_{1}=\partial_{1}'$, $\gamma_{k}=\partial_{1}$, $S_{\gamma_{1}}^{\text{in}}\subsetneq S_{\gamma_{2}}^{\text{in}}\subsetneq\cdots\subsetneq S_{\gamma_{k}}^{\text{in}}$, and $\gamma_{k'}\subset\partial S_{\gamma_{k'+1}}^{\text{in}}$ for each $1\leq k'\leq k-1$ and $k\geq1$(see Figure \ref{fig:16}).
%FIGURE16
\begin{figure}[b]
\center
\includegraphics[width=100mm]{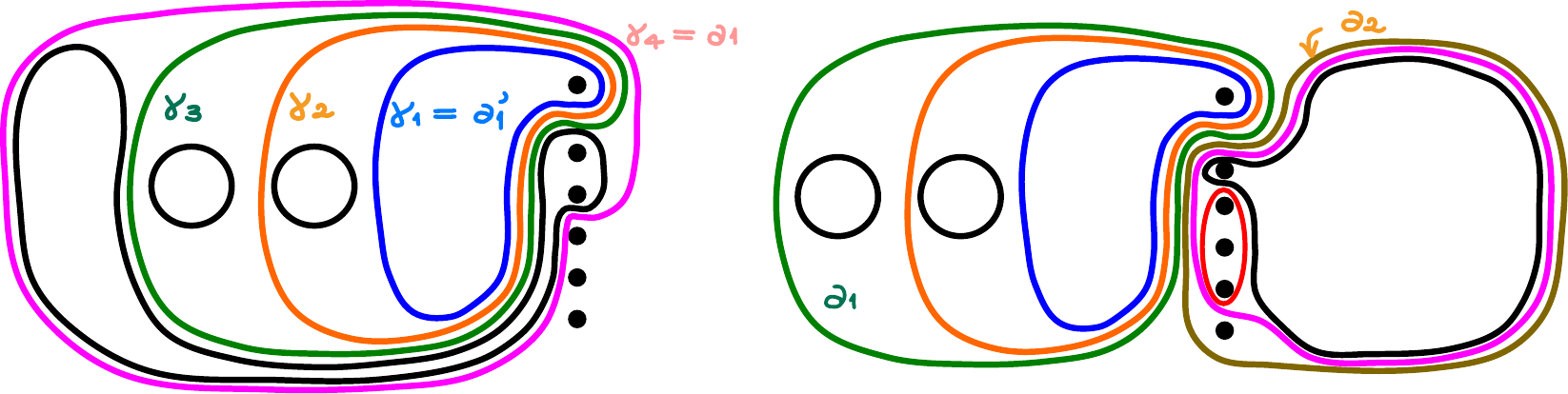}
\caption{An image of $\gamma'$s on the left and the desired curve depicted as the red curve on the right, where the six dots represent the punctures of $\Pi_{T}$.}
\label{fig:16}
\end{figure}
If there exists some integer $k'\in{\{1,\cdots,k\}}$ such that $\gamma_{k'}$ bounds exactly three punctures of $\Pi_{T}$, then the curve is a desired one. Suppose that none of them bounds three punctures of $\Pi_{T}$. If some $\gamma_{k'}$ bounds exactly four or five punctures of $\Pi_{T}$, then we have a desired cut-reducing curve inside $\gamma_{k'}$ as in \textbf{Subcase1a}. Using Lemma \ref{lem:efficient defining pair 3}, if some curve $\gamma_{k'}$ bounds two punctures of $\Pi_{T}$, the curve $\gamma_{k'+1}$ must bound three or five punctures of $\Pi_{T}$, concluding the result. Assume that every curve $\gamma_{k'}$ bound exactly one puncture of $\Pi_{T}$. From Lemma \ref{lem:efficient defining pair 7}, either $\partial_{2}$ or $\partial_{3}$ bounds remaining five punctures of $\Pi_{T}$ and, inside it, we can find the desired cut-reducing curve as in \textbf{Subcase1a} (see Figure \ref{fig:16}).\\
\textbf{Case2} $\Pi_{T}=\Pi_{l}$ for some $l\in{\{1,\cdots,L\}}$.

We show this by induction with respect to $1\leq L<n-1$. If $L=1$, $c$ bounds $\Pi_{T}\cup\{a_{1}^{2},\cdots,a_{m_{2}}^{2}\}$, where $1\leq m_{2}\leq|\Pi_{2}|$. Let $\partial\pi_{c}^{\text{in}}=:c\cup\partial_{1}'\cup\partial_{2}'$. If $m_{2}=1$, $\partial_{1}'$ is reducing, that is, it bounds $\Pi_{T}$ from Lemma \ref{lem:efficient defining pair 4}, and we can find a desired cut-reducing curve inside $\partial_{1}'$ using Lemma \ref{lem:efficient defining pair 3}. If $m_{2}=2$, letting $\partial\pi_{\partial_{1}'}^{\text{in}}=:\partial_{1}'\cup\partial_{1}''\cup\partial_{2}''$, $\partial_{1}''$ is cut-reducing such that it bounds all punctures of $\Pi_{T}$ and one puncture of $\{a_{1}^{2},a_{2}^{2}\}$ from Lemma \ref{lem:efficient defining pair 5}. From the result of $m_{2}=1$, we know that there exists a desired curve inside $\partial_{1}'$. If $m_{2}=3$, $\partial_{1}'$ is a desired one. If $m_{2}=4$, from Lemma \ref{lem:efficient defining pair 5}, there exists a cut-reducing curve $\partial_{1}''$ inside $\partial_{1}'$ such that $\partial_{1}''$ bounds three punctures of $\{a_{1}^{2},a_{2}^{2},a_{3}^{2},a_{4}^{2}\}$. Applying the case of $m_{2}=3$, we find a desired curve inside $\partial_{1}''$. If $m_{2}=5$, letting $\partial_{1}'$ be cut-reducing, $\partial_{1}'$ bounds either one, three, or five punctures of $\{a_{1}^{2},a_{2}^{2},a_{3}^{2},a_{4}^{2},a_{5}^{2}\}$. If it bounds three, then the $\partial_{1}'$ is a desired one. If $\partial_{1}'$ bounds one puncture, then we can find the desired curve inside it applying the case of $m_{2}=1$. If $\partial_{1}'$ bounds five punctures, then the other $\partial_{2}'$ curve bounds exactly six punctures $\Pi_{T}$ and there exists a desired cut-reducing curve inside $\partial_{2}'$ as in the case of $m_{2}=1$. Suppose that $m_{2}=6$, following that both $\partial_{1}'$ and $\partial_{2}'$ are either reducing or cut-reducing from Lemma \ref{lem:efficient defining pair 6}. Suppose now the former case and $\partial_{1}'$ bounds $\Pi_{T}$. Applying the case of $m_{2}=1$, we have the desired curve inside $\partial_{1}'$. Suppose the later case, that is, both curves are cut-reducing. Possible cases of the number of punctures bounded by $\partial_{1}'$ is one, three, or five, and we can find a desired curve in any cases as in the cases of $m_{2}=1$, $3$, or $5$.

Finally, we suppose that there exists a desired cut-reducing curve inside $\partial_{1}'$ for any integer $1\leq L<n-2$. Let $c$ be an essential curve bounding $\bigcup_{l=1}^{L+1}\Pi_{l}\cup\{a_{1}^{L+2},\cdots,a_{m_{L+2}}^{L+2}\}$ and $\partial\pi_{c}^{\text{in}}=:c\cup\partial_{1}'\cup\partial_{2}'$. If $m_{L+2}\in{\{1,2,3,4,5\}}$, then $\partial_{1}'$ and $\partial_{2}'$ are c-reducing and one of them bounds six punctures $\Pi_{T}$, concluding the result using the induction hypothesis. Suppose that $m_{L+2}=6$. In this situation, $c$ is reducing for $\partial\mathcal{D}_{i}$ and both $\partial_{1}'$ and $\partial_{2}'$ are either reducing or cut-reducing from Lemma \ref{lem:efficient defining pair 6}. If they are reducing, then a desired cut-reducing curve is inside one of them using the induction hypothesis. Assume that $\partial_{1}'$ and $\partial_{2}'$ are cut-reducing. If $\partial_{1}'$ does not bound $\Pi_{T}$, then we can find a desired curve inside $\partial_{2}'$ by the induction hypothesis. If $\partial_{1}'$ bounds some odd number of punctures of $\Pi_{T}$, then $\partial_{2}'$ bounds the remaining odd number of punctures of $\Pi_{T}$ and we have a desired cut-reducing curve inside one of them applying the same discussion in \textbf{Case1}. This completes the proof of Lemma \ref{lem:main lemma 1}.

\end{proof}

Finally, we prove the follwing lemma to prove Theorem \ref{thm:main theorem}.

\begin{lem}
Let $\mathcal{T}$ be an unstabilized bridge trisection of a finite distant sum of standard surfaces $F$ whose spine and bridge surface are $\bigcup_{i\not=j}(B_{ij},\alpha_{ij})$ and $\Sigma$, respectively. Let $(p_{ij}^{i},p_{ki}^{i})\in{P_{c}(\alpha_{ij})\times P_{c}(\alpha_{ki})}$ be an efficient defining pair for the unlink $\partial\mathcal{D}_{i}=\alpha_{ij}\cup\alpha_{ki}$ for each $\{i,j,k\}=\{1,2,3\}$. Then, 
$d^{*}(p_{ij}^{i},p_{ki}^{i})\geq n(F)$ holds for each $\{i,j,k\}=\{1,2,3\}$, where $n(F)$ denotes the number of standard tori included in $F$.
\label{lem:main lemma 2}
\end{lem}

\begin{proof}
Suppose $n(F)>0$. From Lemma \ref{lem:main lemma 1}, there exist $n(F)$ cut-reducing curves $\gamma_{1}^{i},\cdots,\gamma_{k_{\mathcal{T}}}^{i}\in{\psi_{i}}$ satisfying the properties (i), (ii), and (iii) of the lemma. Suppose that $\gamma_{l}^{i}$ is not moved by a shortest path $p_{ij}^{i}\rightarrow p_{ij}^{j}$ for some integer $l\in{\{1,\cdots,k_{\mathcal{T}}\}}$. Then, $\gamma_{l}^{i}\in{p_{ij}^{i}\cap p_{ij}^{j}}$ and this curve is not moved by paths $p_{ij}^{i}\rightarrow p_{ki}^{i}$ and $p_{ij}^{j}\rightarrow p_{jk}^{j}$ from Lemma \ref{lem:efficient defining pair 1} and the curve is c-reducing for each tangle $\alpha_{12}$, $\alpha_{23}$, and $\alpha_{31}$. Considering $\gamma_{l}^{i}$ as the cut-reducing curve for each tangle forming the spine of the unstabilized bridge trisection of the standard unknotted torus $T_{l}$, this implies that the bridge trisection of $T_{l}$ is reducible or stabilized from Lemma \ref{lem:reducible and stabilized}, and this is a contradiction.
\end{proof}

\subsection{Upper bounds for the $\mathcal{L}$-invariants.}

In this subsection, we show some figures where we can see a upper bound for $\mathcal{L}(F)$, where $F$ is a finite distant sum of the standard unknotted surfaces.

\begin{lem}
$\mathcal{L}(F)\leq3n(F)$ holds, where $F$ is a finite distant sum of the standard surfaces. Consequently, $\mathcal{L}^{*}(F)\leq3n(F)$ holds.
\label{lem:upper bound}
\end{lem}

\begin{proof}
The above inequality holds from Figure \ref{fig:efficient1}, Figure \ref{fig:efficient2}, and Figure \ref{fig:efficient3}.
\end{proof}

\begin{figure}[h]
\centering
\includegraphics[width=113mm]{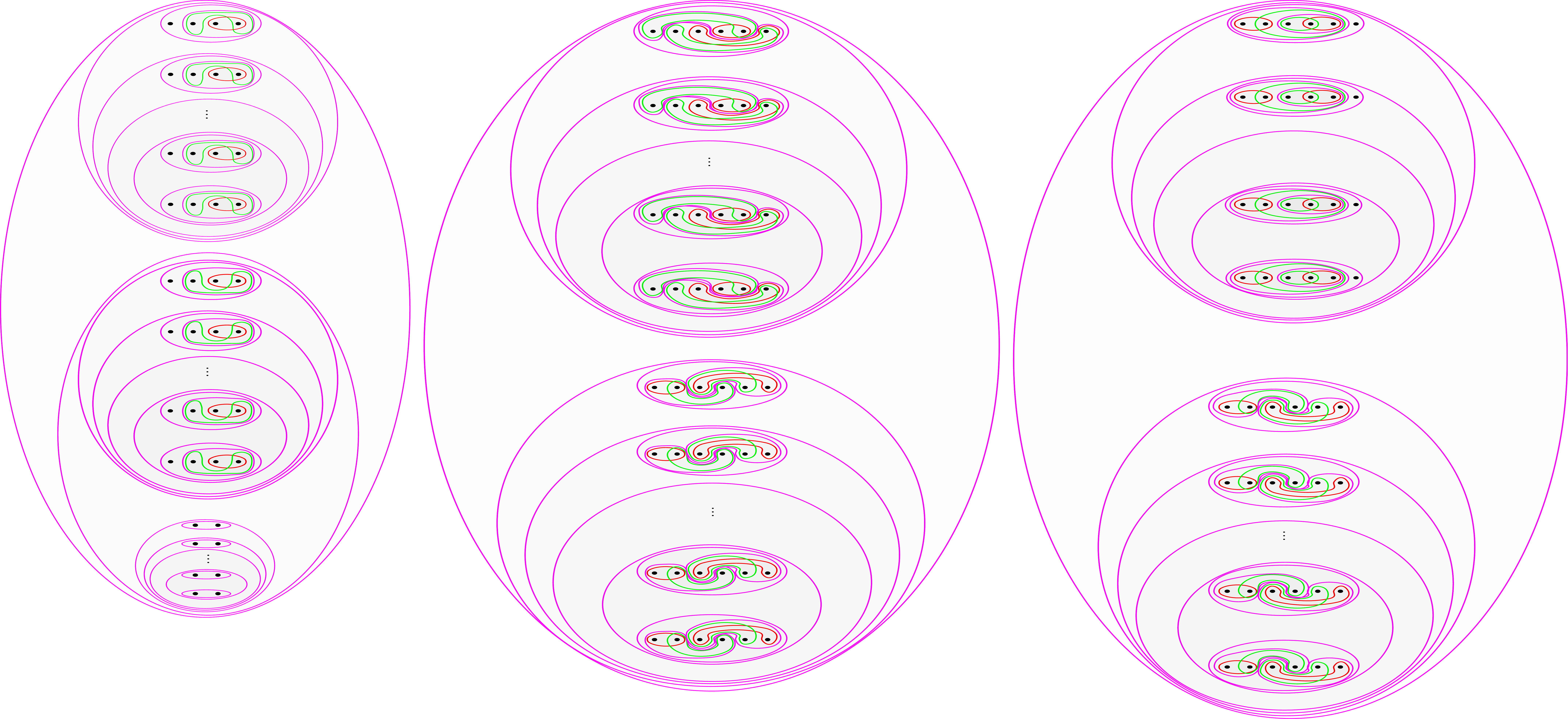}
\caption{An efficient defining pair $(p_{12}^{1},p_{31}^{1})$.}\label{fig:efficient1}
\end{figure}

\begin{figure}[h]
\centering
\includegraphics[width=113mm]{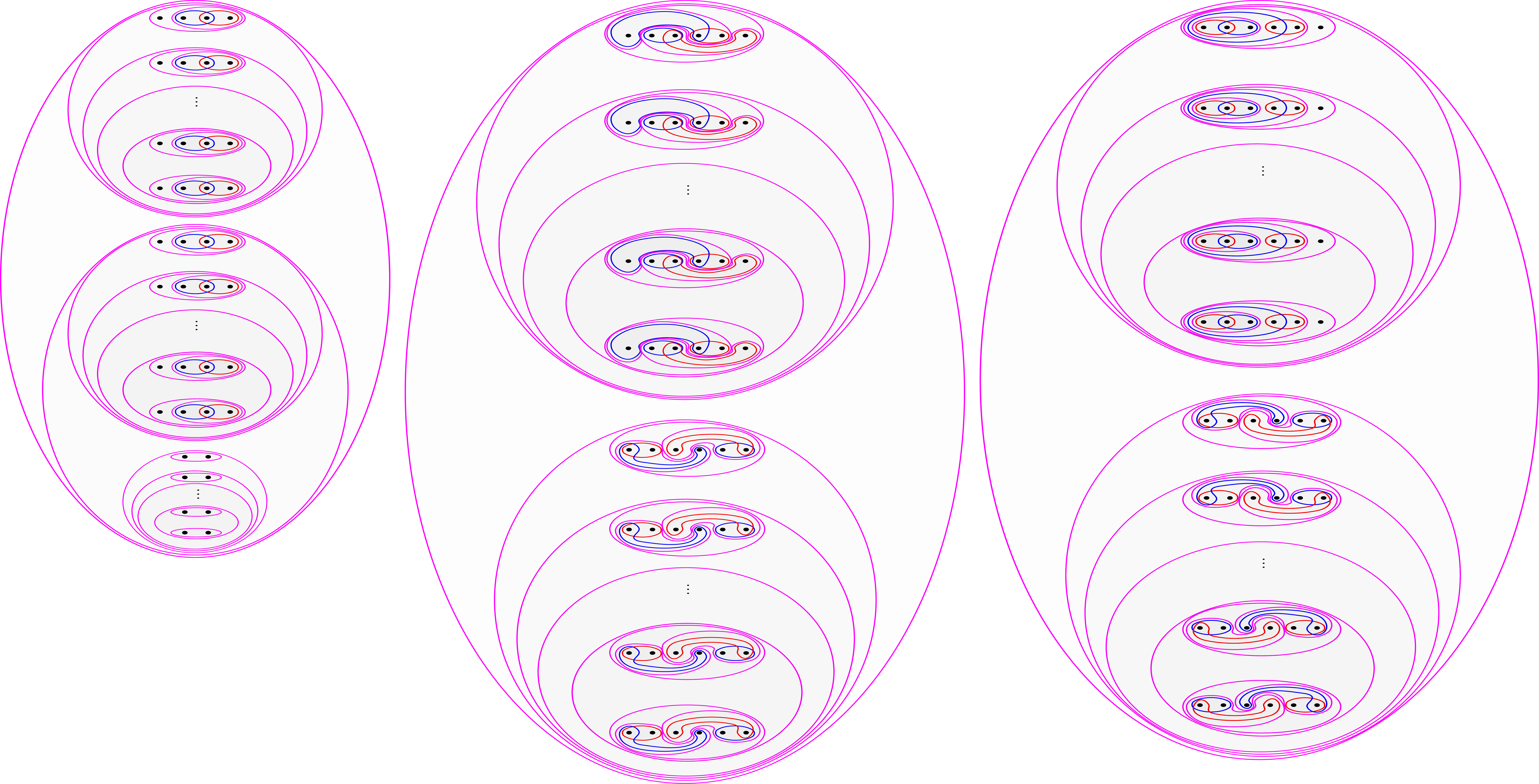}
\caption{An efficient defining pair $(p_{23}^{2},p_{12}^{2})$.}\label{fig:efficient2}
\end{figure}

\begin{figure}[h]
\center
\includegraphics[width=113mm]{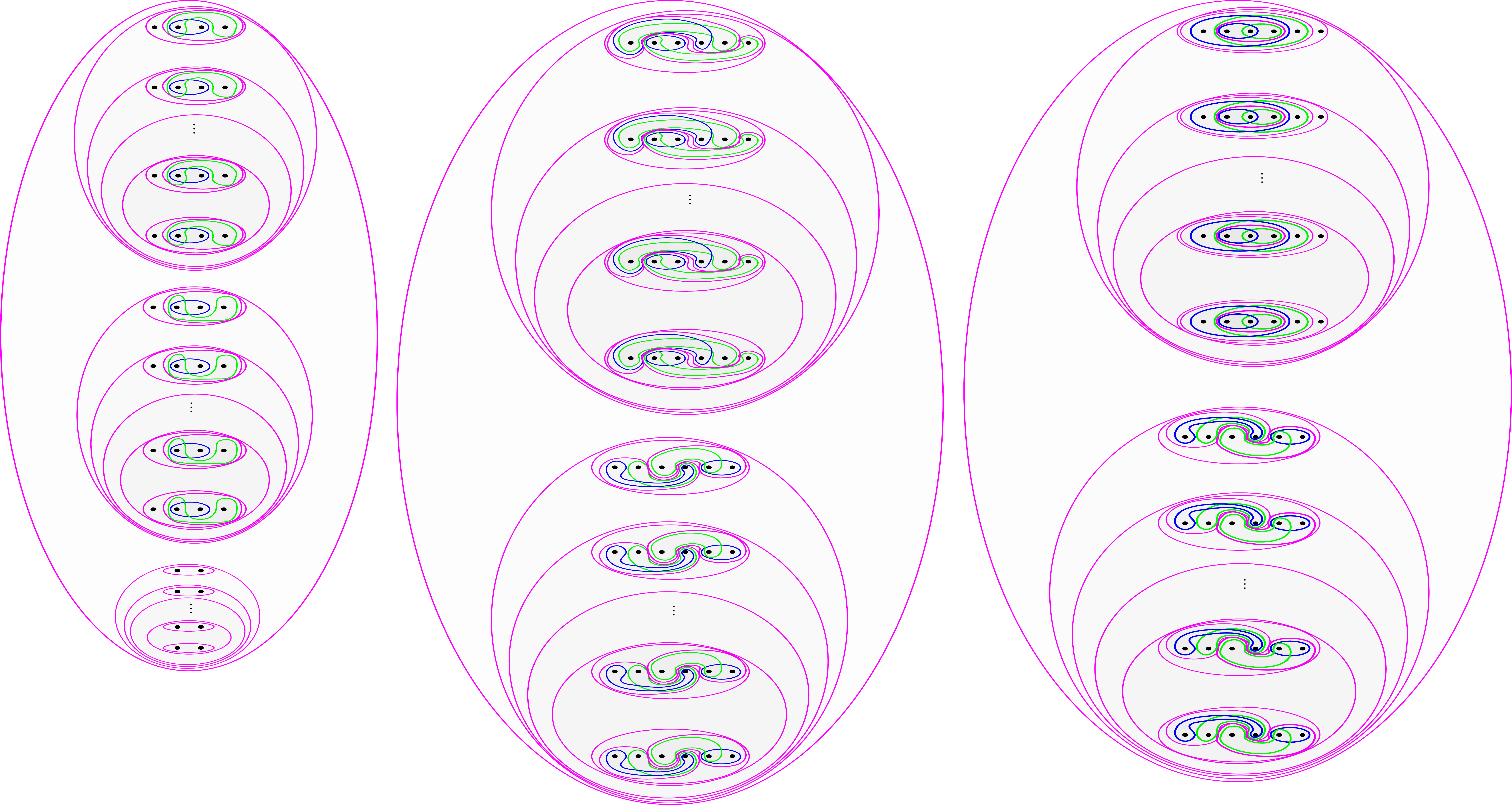}
\caption{An efficient defining pair $(p_{31}^{3},p_{23}^{3})$.}\label{fig:efficient3}
\end{figure}

%\bibliographystyle{alpha}
%\bibliography{bibliography.bib}

\end{document}